\documentclass[12pt, a4paper]{article}
\usepackage{latexsym, amssymb, amsfonts, amsmath, amsthm}
\usepackage{graphicx}
\usepackage{caption}
\usepackage{subcaption}
\usepackage{hyperref}
\usepackage{float}
\usepackage{fullpage}

\usepackage[round]{natbib}
\renewcommand{\cite}{\citet}

\newcommand{\dhat}[1]{\tilde{#1}}

\usepackage{algorithm}
\usepackage{algpseudocode}
\algrenewcommand{\algorithmicrequire}{\textbf{Input:}}
\algrenewcommand{\algorithmicensure}{\textbf{Output:}}

\newcommand{\stsets}[1]{\mathbb{#1}}
\newcommand{\A}{\stsets{A}}
\newcommand{\R}{\stsets{R}}
\renewcommand{\L}{\stsets{L}}
\newcommand{\M}{\stsets{M}}

\newcommand{\N}{\stsets{N}}
\newcommand{\Z}{\stsets{Z}}

\newcommand{\T}{\stsets{T}}

\newtheorem{theorem}{Theorem}

\renewcommand{\P}{\mathbf{P}}

\DeclareMathOperator{\E}{{\bf E}}

\DeclareMathOperator{\one}{{ 1\hspace*{-0.55ex}I}}
\newcommand{\cond}{\hspace*{1ex}
  \rule[-1ex]{0.15ex}{3ex}\hspace*{1ex}}

\newcommand{\bN}{\boldsymbol{N}}

\newcommand{\Exp}{\ensuremath{\mathsf{Exp}}}
\newcommand{\Po}{\ensuremath{\mathsf{Po}}}

\newcommand{\dtv}{d_{\mathrm{TV}}}

\newcommand{\thru}{,\dotsc,}

\renewcommand{\epsilon}{\varepsilon}
\renewcommand{\phi}{\varphi}

\newcommand{\rd}{\mathrm{d}} 

\newcommand{\lchf}{L_{\mathrm{ChF}}}
\newcommand{\lcof}{L_{\mathrm{CoF}}}

\newlength{\querylen}
\usepackage{fancybox}

\begin{document}
\title{Nonparametric estimation of infinitely divisible distributions based on
  variational analysis on measures} 

\author{Alexey Lindo\thanks{Department of Mathematical Sciences,
    Chalmers University of Technology and University of Gothenburg,
    Sweden. Email: 
    \texttt{sergei.zuev@chalmers.se}} \and
  \addtocounter{footnote}{-1}Sergei Zuyev\footnotemark \and
  \addtocounter{footnote}{-1}Serik Sagitov\footnotemark} 

\date{\today}
\maketitle

\begin{abstract}
The paper develops new methods of non-parametric estimation a compound Poisson distribution. Such a problem arise, in particular, in the inference of a L\'evy process recorded at equidistant time intervals. Our key estimator is based on series decomposition of functionals of a measure and relies on the steepest descent technique recently developed in variational analysis of measures. Simulation studies demonstrate applicability domain of our methods and how they positively compare and complement the existing techniques. They are particularly suited for discrete compounding distributions, not necessarily concentrated on a grid nor on the positive or negative semi-axis. They also give good results for continuous distributions provided an appropriate smoothing is used for the obtained atomic measure.



\end{abstract}

\section{Introduction}\label{sec:intro}
The paper develops new methods of non-parametric estimation of the
distribution of compound Poisson data. Such data naturally arise
in the inference of a \emph{L\'evy process} which is a stochastic
process $(W_{t})_{t\ge0}$ with $W_0=0$ and time homogeneous independent
increments. Its characteristic function necessarily has the
form $Ee^{i\theta W_t}=e^{t\psi(\theta)}$ with
\begin{equation}\label{eq:LK}
  \psi(\theta)= 
    ia\theta-\frac{\sigma^2\theta^2}{2} + \int (e^{i\theta
      x}-1-i\theta x\one_{\{|x|<\epsilon\}}) \Lambda(\rd x),
\end{equation}
where $\epsilon>0$ is a fixed positive number, $a\in\R$ is a
\emph{drift} parameter, $\sigma^2\in[0,\infty)$ is the \emph{variance}
of the \emph{Brownian motion} component, and the so-called
\emph{L\'evy measure} $\Lambda $ satisfying
\begin{equation}\label{eq:levy_constraints}
  \Lambda(\{0\})=0,\quad \int \min\{1,x^2\}\, \Lambda(\rd x) < \infty.
\end{equation}
Here and below the integrals are taken over the whole $\R$ unless specified otherwise.
In a special case with $\sigma=0$ and
$\int _{(-\epsilon,\epsilon)}|x| \Lambda(\rd x) < \infty$, we get a \emph{pure jump} L\'evy
process characterised by
\begin{equation}\label{La}
\psi(\theta)= \int (e^{i\theta x}-1) \Lambda(\rd x),
\end{equation}
or equivalently,
\begin{displaymath}
  \psi(\theta)= 
  ia\theta+ \int  (e^{i\theta x}-1-i\theta x\one_{|x|<\epsilon})
  \Lambda(\rd x),\quad a=\int_{(-\epsilon,\epsilon)} x\Lambda(\rd x). 
\end{displaymath}
In an even more restrictive case with a finite total mass
$\|\Lambda\|:=\Lambda(\R)$, the L\'evy process becomes a \emph{compound
Poisson process} with times of jumps being a Poisson process with
intensity $\|\Lambda\|$, and the jump sizes being independent random
variables with distribution $\|\Lambda\|^{-1}\Lambda(dx)$. Details can
be found, for instance, in~\cite{SatoK1999}.

Suppose the L\'evy process is observed at regularly spaced times
producing a random vector $(W_0, W_h,W_{2 h},\ldots,W_{n h})$ for some
time step $h>0$. The consecutive increments $X_i=W_{i h}-W_{(i-1) h}$
then form a vector $(X_{1}, \ldots X_{n})$ of independent random
variables having a common infinitely divisible distribution with the
characteristic function $\phi(\theta)=e^{h\psi(\theta)}$, and thus can be used to
estimate the distributional triplet $(a,\sigma,\Lambda)$ of the
process.  Such inference problem naturally arises in in financial
mathematics~\cite{Cont2003}, queueing theory~\cite{Asmussen2008},
insurance~\cite{Mikosch2009} and in many other situations, where
L\'{e}vy processes are used.

By the L\'{e}vy-It\^{o} representation theorem~\cite{SatoK1999}, every
L\'{e}vy process is a superposition of a Brownian motion with drift
and a square integrable pure jump martingale. The latter can be
further decomposed into a pure jump martingale with the jumps not
exceeding in absolute value a positive constant $\epsilon$ and a
compound Poisson process with jumps $\epsilon$ or above. In practice,
only a finite increment sample $(X_1\thru X_n)$ is available, so there
is no way to distinguish between the small jumps and the Brownian
continuous part. Therefore one usually chooses a threshold level
$\epsilon>0$ and attributes all the small jumps to the Brownian
component, while the large jumps are attributed to the compound
Poisson process component (see, e.g.~\cite{AsmRos:01} for an account
of subtleties involved).

Provided an estimation of the continuous and the small jump part is
done, it remains to estimate the part of the L\'evy measure outside of the
interval $(-\epsilon,\epsilon)$. Since this corresponds to the
compound Poisson case, estimation of such $\Lambda$ is usually 
called \emph{decompounding} which is the main object of study in this paper.

Previously developed methods include discrete decompounding approach
based on the inversion of Panjer recursions as proposed
in~\cite{Buchman2003}.  \cite{Comte2014},
\cite{Duval2013} and \cite{Es2007} studied the continuous
decompounding problem when the measure $\Lambda$ is assumed to have a
density. They apply Fourier inversion in combination with kernel
smoothing techniques for estimating an unknown density of the L\'evy
measure. In contrast, we do not distinguish between discrete and
continuous $\Lambda$ in that our algorithms based on direct
optimisation of functionals of a measure work for both situations on a
discretised phase space of $\Lambda$. However, if one sees many small
atoms appearing in the solution which fill a thin grid, this may
indicate that the true measure is absolutely continuous and some kind of
smoothing should yield its density.

Specifically, we propose a combination of two non-parametric methods for estimation
of the L\'evy measure which we call Characteristic Function Fitting
(ChF) and Convolution Fitting (CoF). ChF
deals with a general class of L\'evy processes, while CoF more
specifically targets the pure jump L\'evy process characterised by
\eqref{La}.

The most straightforward approach is to use the moments fitting,
see \cite{FeuMcD:81b} and \cite{CarFlo:00}, or the empirical distribution function
 \begin{equation*}
\hat  F_{n}(x) = \frac{1}{n} \sum_{k = 1}^{n} \one_ {\{X_{k} \le  x\}}
\end{equation*}
to infer about the triplet $(a, \sigma, \Lambda)$. Estimates can be
obtained by maximising the likelihood ratio (see, e.g.\
\cite{QinLaw:94}) or by minimising some measure of proximity between
$F(x)$ and $\hat{F}_{n}(x)$,
where the dependence on $(a, \sigma, \Lambda)$ comes through $F$
via the inversion formula of the characteristic function:
\begin{displaymath} 
  F(x) - F(x - 0) = \frac{1}{2 \pi} \lim_{y \to \infty} \int_{-y}^{y}
  \exp\{h\psi(\theta)-i \theta x\} \rd \theta.
\end{displaymath}
For the estimation, the characteristic function in the integral above
is replaced by the empirical characteristic function:
\begin{equation*}
  \hat\phi_{n}(\theta) = \frac{1}{n} \sum_{k = 1}^{n} e^{i \theta X_{k}}.
\end{equation*}
Algorithms based on the inversion of the empirical characteristic function and on the relation
between its derivatives were proposed in~\cite{WatKul:03}. For a
comparison between different estimation methods, see a recent survey
\cite{SueNis:05}. Note that inversion of the empirical characteristic function, in
contrast to the inversion of its theoretical counterpart, generally
leads to a complex valued measure which needs to be dealt with.

Instead, equipped with the new theoretical and numeric optimisation
methods developed recently for functionals of measures
(see~\cite{MolchanovI2002} and the references therein), we use the
empirical characteristic function directly: the ChF estimator for the
compounding measure $\Lambda$ or, more generally, of the whole triplet
$(a, \sigma, \Lambda)$ may be obtained by minimisation of the loss
functional
\begin{equation}\label{eq:lchf}
  \lchf(a, \sigma,\Lambda) = \int |e^{ h\psi(\theta)} -
  \hat\phi_{n}(\theta) |^{2}  \omega(\theta)\rd \theta, 
\end{equation}
where $\psi(\theta)\equiv\psi(\theta; a, \sigma,\Lambda) $ is given by
\eqref{eq:LK} and $\omega(\theta)$ is a weight function. Typically
$\omega(\theta)$ is a positive constant for
$\theta\in[\theta_1,\theta_2]$ and zero otherwise, but it can also be
chosen to grow as $\theta \to 0$, this would emphasise a better
agreement with the estimated distribution for smaller jumps.

\emph{Parametric} inference procedures based on the empirical characteristic function has been
known for some time, see, e.g., \cite{FeuMcD:81a} and the references
therein. Our main contribution is that we compute explicitly the
derivative of the loss functional~\eqref{eq:lchf} with respect to the \emph{measure}
$\Lambda$ and perform the steepest descent directly on the cone of
non-negative measures to a local minimiser.  It must be noted that, as
a simple example reveal, the functionals based on the empirical characteristic function usually
have a very irregular structure, see Figure~\ref{fig:irregular}. As a
result, the steepest descent often fails to attend the global optimal
solution, unless the starting point of the optimisation procedure is
carefully chosen.
\begin{figure}[H]
\begin{subfigure}{.5\textwidth}
  \centering
  \includegraphics[width=.95\linewidth]{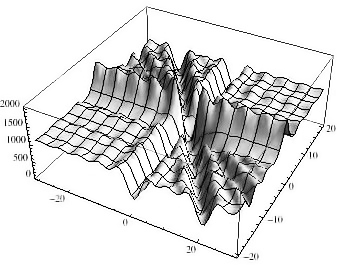}
\end{subfigure}%
\begin{subfigure}{.5\textwidth}
  \centering
  \includegraphics[width=.95\linewidth]{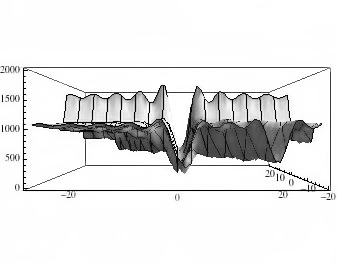}
\end{subfigure}
\caption{Illustration of intrinsic difficulties faced by any
  characteristic function fitting procedure. Suppose one knows the
  data are coming from a shifted Poisson distribution (i.e.\
  $\sigma=0$ and the L\'evy measure has the form
  $\Lambda=x\delta_{1}$). Assume now a big data sample is actually
  coming from a Poisson $\Po(1)$ distribution so that $\hat{\phi}_n$
  in \eqref{eq:lchf} is very close to
  $\exp\{e^{i\theta}-1)\}$. Plotted are the values of~\eqref{eq:lchf}
  with shift parameter $-20\leq a\leq 20$ and $0\leq x\leq20$. It is
  seen that any optimisation algorithm would have difficulties
  converging to the global minimum $a=0,\ x=1$ even in this simple two
  parameter case.
  \label{fig:irregular}}
\end{figure}

In contrast, the proposed CoF estimation method is not using the
empirical characteristic function and is based on Theorem~\ref{Main} below which presents the
convolution
\begin{displaymath}
 F^{*2}(x) = \int F(u) F(y- u) du
\end{displaymath}
as a functional of $\Lambda$. It has an explicit form of an infinite
Taylor series in direct products of $\Lambda$, but truncating it to
only the first $k$ terms we build a loss function $\lcof^{(k)}$ by
comparing two estimates of $F^{*2}$: the one based on the truncated
series and the other being the empirical convolution
$F^{2*}_{n}$. CoF is able to produce nearly optimal estimates
$\hat \Lambda _{k}$ when large values of $k$ are taken, but this also
drastically increases the computation time.

A practical combination of these methods recommended by this paper is
to find $\hat \Lambda _{k}$ using CoF with a low value of $k$, and
then apply ChF with $\hat \Lambda _{k}$ as the starting value. The
estimate for such a two-step procedure will be denoted by
$\dhat{\Lambda}_{k}$ in the sequel.

To give an early impression of our approach, let us demonstrate the
performance of the ChF methods on the famous data by Ladislaus Bortkiewicz
who collected the numbers of Prussian soldiers killed by a horse kick
in 10 cavalry corps over a 20 years period \cite{Bort:98}. The counts
0, 1, 2, 3, and 4 were observed 109, 65, 22, 3 and 1 times, with
0.6100 deaths per year and cavalry unit. The author argues that the
data are Poisson distributed which corresponds to the measure
$\Lambda$ concentrated on the point $\{1\}$ (only jumps of size 1) and
the mass being the parameter of the Poisson distribution which is then
estimated by the sample mean 0.61. Figure~\ref{horse} on its left panel presents
the estimated L\'evy measures for the cutoff values $k=1,2,3$ when
using CoF method. For the values of $k=1,2$, the result is a measure having many
atoms. This is explained by the fact that the accuracy of the
convolution approximation is not enough for this data, but $k=3$
already results in a measure $\hat\Lambda_3$ essentially
concentrated at $\{1\}$ thus supporting the Poisson model with
parameter $\|\hat\Lambda_3\|=0.6098$. In Section~\ref{Ssd} we return
to this example and explain why the choice of $k=3$ is reasonable.
We observed that the convergence of the ChF method 
depends critically on the choice of the initial measure, especially on
its total mass. However, the proposed combination of CoF followed by
ChF demonstrates (the right plot) that this two-step (faster)
procedure results in the estimate $\dhat{\Lambda}_{1}$ which is as
good as $\hat\Lambda_3$.

The rest of the paper has the following structure.  Section~\ref{Sopt}
introduces the theoretical basis of our approach -- a constraint optimisation
technique in the space of measures.  In Section~\ref{SM1} we perform
analytic calculations of the gradient of the loss functionals needed
for the implementation of ChF.  Section~\ref{SM3} develops
the necessary ingredients for the CoF method and proves the main
analytical result of the paper, Theorem~\ref{Main}.
In Section~\ref{sec:implement} we give some details on the implementation
of our algorithms in R-language.  Section~\ref{Ssd}
contains a broad range of simulation results illustrating
performance of our algorithms. We conclude by
Section~\ref{sec:discussion}, where we summarise our approach and give
some practical recommendations.

\begin{figure}[H]
\begin{subfigure}{.5\textwidth}
  \centering
  \includegraphics[width=.95\linewidth]{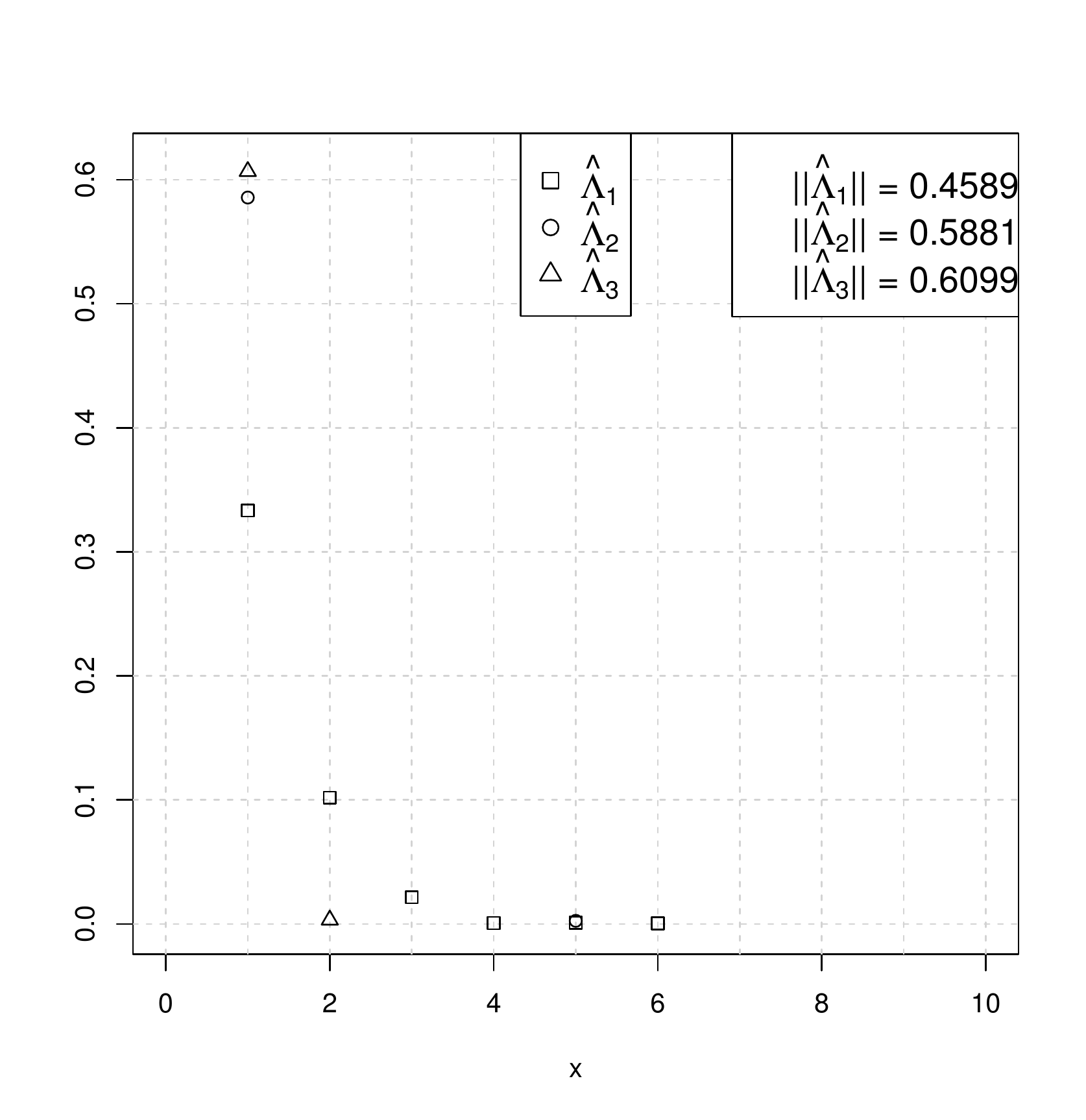}
\end{subfigure}%
\begin{subfigure}{.5\textwidth}
  \centering
  \includegraphics[width=.95\linewidth]{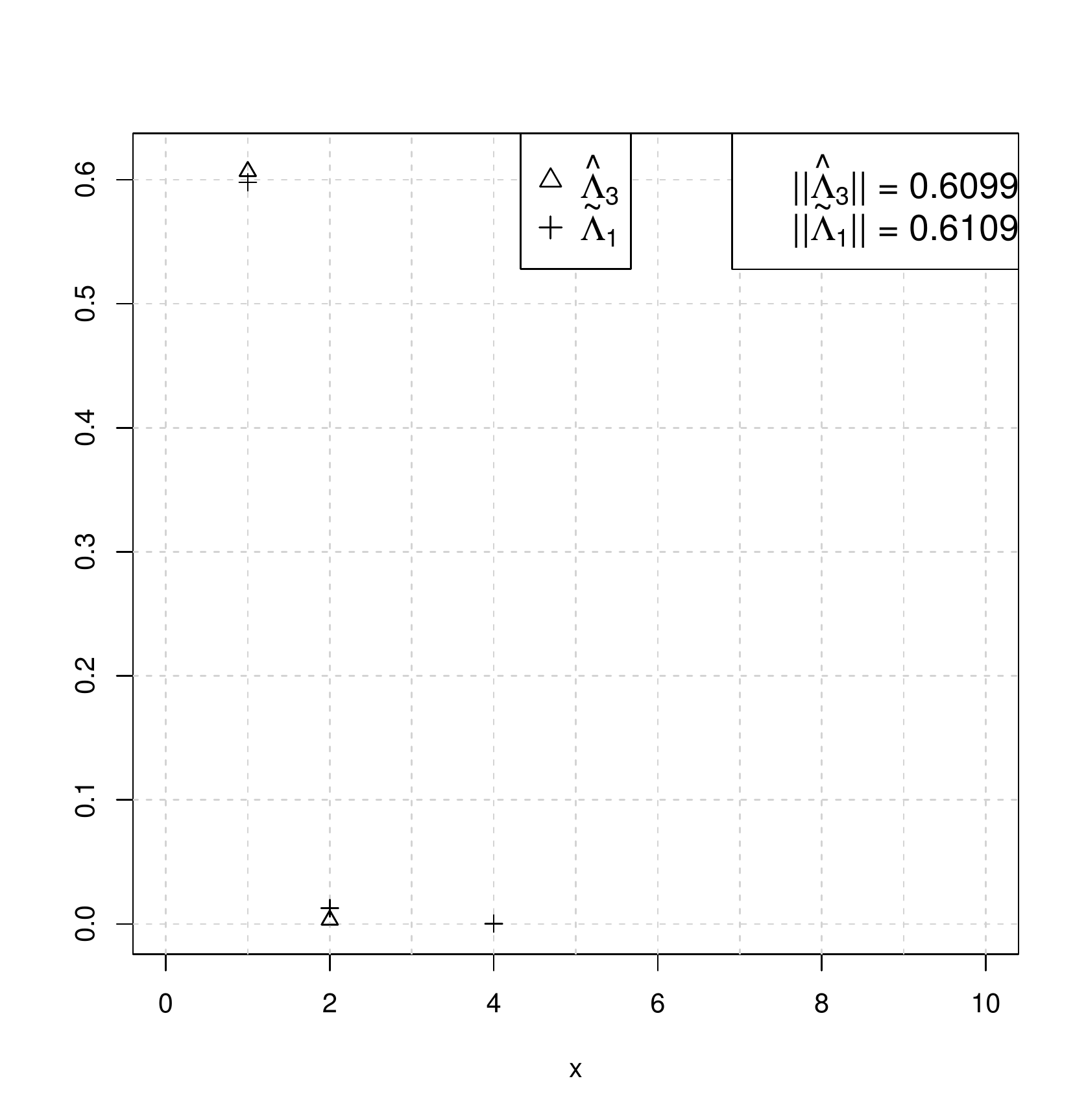}
\end{subfigure}
\caption{The Bortkiewicz horse kick data. Left panel: comparison of
  CoF estimates for $k=1,2,3$. Right panel: comparison of the estimate
  by CoF with $k=3$ and a combination of  CoF with $k=1$ followed by
  ChF. } 
\label{horse}
\end{figure}

\section{Optimisation in the space of measures.}
\label{Sopt}

In this section we briefly present the main ingredients of the
constrained optimisation of functionals of a measure. Further details
can be found in \cite{MolchanovI2000a} and \cite{MolchanovI2000}. 

In this paper we are dealing with measures defined on the Borel
subsets of $\mathbb{R}$. Recall that any \emph{signed} measure $\eta$ can be
represented in terms of its Jordan decomposition:
$\eta = \eta^{+} + \eta^{-}$, where $\eta^{+}$ and $\eta^{-}$ are
orthogonal non-negative measures. The total variation norm is then
defined to be $\|\eta\| = \eta^{+}(\R) + \eta^{-}(\R)$. Denote by $\M$
and $\M_+$ the class of signed, respectively, non-negative measures
with a finite total variation. The set $\M$ then becomes a Banach
space with sum and multiplication by real numbers defined set-wise:
$(\eta_1+\eta_2)(B):=\eta_1(B)+\eta_2(B)$ and $(t\eta)(B):=t\eta(B)$
for any Borel set $B$ and any real $t$. The set $\M_+$ is a
\emph{pointed cone} in $\M$ meaning that the zero measure is in $\M_+$
and that $\mu_1+\mu_2\in\M_+$ and $t\mu\in\M_+$ as long as
$\mu_1,\mu_2,\mu\in\M_+$ and $t\geq 0$.

A functional $G:\M\mapsto \R$ is called \emph{Fr\'echet} or
\emph{strongly differentiable} at $\eta\in\M$ if there exists a
bounded linear operator (a \emph{differential}) $DG(\eta)[\cdot]:\ \M\mapsto \R$ such that 
\begin{equation}
\label{eq:defdif}
  G(\eta+\nu)-G(\eta)=DG(\eta)[\nu]+o(\|\nu\|)\ \nu\in\M.  
\end{equation}
If for a given $\eta\in\M$ there exists a bounded function $\nabla
G(\,\cdot\,;\eta)\,:\,\mathbb{R}\to \mathbb{R}$ such that
\begin{displaymath}
  DG(\eta)[\nu]=\int \nabla G(x;\eta)\,\nu(\rd x)\ \text{for all}\ \nu\in\M,
\end{displaymath}
then such $\nabla G(x;\eta)$ is called the \emph{gradient function} for $G$ at
$\eta$. Typically, and it is the case for the functionals of measure we
consider here, the gradient function does exist so that the
differentials do have an integral form.

For example, an integral of a bounded function
$G(\eta)=\int f(x)\eta(dx)$ is already a bounded linear functional of
$\eta$ so that $\nabla G(x;\eta)=f(x)$ for any $\eta$. More generally, for
a composition $G(\eta)=v(\int f(x)\eta(\rd x))$, where $v$ is a
differentiable function, the gradient function can be obtained by the
Chain rule:
\begin{equation}
  \label{eq:chain}
  \nabla G(x;\eta)=v'\Bigl(\int f(y)\,\eta(\rd y)\Bigr)\,f(x).
\end{equation}
The functional $G$ for this example is strongly differentiable if both functions
$v'$ and $f$ are bounded.

The estimation methods we develop here are based on minimisation of
various loss functions over the class of possible L\'evy measures with
a finite mass. Specifically, we consider minimisation of a strongly
differentiable functional
\begin{equation}
  \label{eq:1}
  L(\Lambda)\to \min\quad \text{subject to}\ \Lambda\in\M_+,\
  H(\Lambda)\in C,
\end{equation}
where the last constraint singles out the set of L\'evy
measures, i.e.\ the measures
satisfying~\eqref{eq:levy_constraints}. This corresponds to taking
$C=\{0\}\times\R$ being a cone in $\R^2$ and
\begin{equation}
  \label{eq:2}
  H(\Lambda)=\Bigl(\Lambda(\{0\}),\int \min\{1,x^2\}\,\Lambda(\rd x)\Bigr).
\end{equation}

\begin{theorem}
  \label{Th:NecessaryConditions}
  Suppose $L\,:\,\mathbb{M} \to \mathbb{R}$ is strongly differentiable
  at a positive finite measure $\Lambda$
  satisfying~\eqref{eq:levy_constraints} and possess a gradient function
  $\nabla L(x;\Lambda)$. If such $\Lambda$ provides a local minimum of $L$ over
  $\mathbb{M}_{+} \cap H^{-1}(C)$, then
    \begin{equation}\label{eq:KT}
    \begin{cases}
      \nabla L (x;\Lambda) \ge 0 \quad \text{for all}\ x \in \mathbb{R}\setminus\{0\}, \\
      \nabla L (x;\Lambda) = 0 \quad \Lambda-a.e. \\
    \end{cases}
  \end{equation}
\end{theorem}

\begin{proof}
First-order necessary criteria for constrained optimisation in a Banach
space can be derived in terms of tangent cones.
Let $\A$ be a subset of $\M$ and $\eta\in\A$. The \emph{tangent cone}
to $\A$ at $\eta$ is the following subset of $\M$:
\begin{displaymath}
  \T_{\A}(\eta) = \liminf_{t \downarrow 0} t^{-1}(\A-\eta).
\end{displaymath}
Recall that the $\liminf_n A_n$ for a family of subsets
$(A_n)$ in a normed
space is the set of the limits of all converging sequences $\{a_n\}$ such that
$a_n\in A_n$ for all $n$.  Equivalently, $\T_{\A}(\eta)$ is the closure of the set
of such $\nu\in\M$ for which there exists an $\epsilon=\epsilon(\nu)>0$ such that
$\eta + t\nu \in \A$ for all $0\leq t\leq\epsilon$.

By the definition of the tangent cone, if $\eta$ is a point of minimum
of a strongly differentiable function $G$ over a set $\A$ then one
must have
\begin{equation}\label{eq:3}
  DG(\eta)[\nu]\geq 0\quad \text{for all}\ \nu\in \T_{\A}(\eta).
\end{equation}
Indeed, assume that there exists $\nu\in \T_{\A}(\eta)$ such that
$DG(\eta)[\nu]:=-\epsilon<0$. Then there is a sequence of positive numbers
$t_n\downarrow 0$ and a sequence $\eta_n\in\A$ such that $\nu=\lim_n
t_n^{-1}(\eta_n-\eta)$ implying $\eta_n\to\eta$ because
$\|\eta-\eta_n\|=t_n(1+o(1))\|\nu\|\to0$. Since any bounded linear
operator is continuous, we also have
\begin{displaymath}
  DG(\eta)[\nu]=DG(\eta)[\lim_n t_n^{-1}(\eta_n-\eta)]
  =\lim_n t_n^{-1} DG(\eta)[\eta_n-\eta]=-\epsilon.
\end{displaymath}
Furthermore, by~\eqref{eq:defdif},
\begin{displaymath}
  DG(\eta)[\eta_n-\eta]=G(\eta_n)-G(\eta)+o(\|\eta-\eta_n\|)=G(\eta_n)-G(\eta)+o(t_n),
\end{displaymath}
thus
\begin{displaymath}
  G(\eta_n)-G(\eta)= -t_n\epsilon(1+o(1))<-t_n\epsilon/2
\end{displaymath}
for all sufficiently small $t_n$. Thus in any ball of $\eta$ there is
a $\eta_n\in\A$ such that $G(\eta_n)<G(\eta)$ so that $\eta$ is not a
point of a local minimum of $G$ over $\A$.

Next step is to find a sufficiently rich class of measures belonging
to the tangent cone to the set $\L:=\M_+\cap H^{-1}(C)$ of all possible
L\'evy measures. For this, notice that for any $\Lambda\in\L$, the Dirac
measure $\delta_x$ belongs to
$\T_{\L}(\Lambda)$ since $\Lambda+t\delta_x\in \L$ for any $t\geq0$ as
soon as $x\neq0$. Similarly, given any Borel $B\subset \R$, the
negative measure $-\Lambda|_B=-\Lambda(\,\cdot\,\cap B)$, which is the
restriction of $-\Lambda$ onto $B$, is also in the tangent cone
$\T_\L(\Lambda)$, because for any $0\leq t\leq 1$ we have
$\Lambda-t\Lambda_B\in\L$.

Since $\nabla G(x;\Lambda)$ is a gradient function, the necessary
condition \eqref{eq:3} becomes
\begin{displaymath}
  \int \nabla G(x;\Lambda)\,\nu(\rd x)\geq 0\quad \text{for all}\
  \nu\in \T_{\L}(\Lambda) 
\end{displaymath}
and substituting $\nu=\delta_x$ above we immediately obtain the
inequality in~\eqref{eq:KT}. Finally, taking $\nu=-\Lambda_B$ yields
\begin{displaymath}
  \int_B \nabla G(x;\Lambda)\,\Lambda(\rd x)\leq 0.
\end{displaymath}
Since this is true for any Borel $B$, then $G(x;\Lambda)\leq0$
$\Lambda$-almost everywhere which, combined with the previous
inequality, gives the second identity  in~\eqref{eq:KT}.

\end{proof}

A rich class of functions of a measure represent the expectation of a
functionals of a Poisson process.

Let $\bN$ be the space of locally finite counting measures $\phi$ on a
Polish space $X$ which will be a subset of an Euclidean space in this
paper. Let $\mathcal{N}$ be the smallest $\sigma$-algebra which makes
all the mappings $\phi\mapsto \phi(B)\in\Z_+$ for $\phi\in\bN$ and
compact sets $B$ measurable. A Poisson point process with the
\emph{intensity measure} $\mu$ is a measurable mapping
$\Pi$ from some probability space into $[\bN,\mathcal{N}]$ such that
for any finite family of disjoint compact sets $B_1\thru B_k$, the
random variables $\Pi(B_1)\thru \Pi(B_k)$ are independent and each
$\Pi(B_i)$ following Poisson distribution with parameter
$\mu(B_i)$. We use notation $\Pi\sim\mathrm{PPP}(\mu)$. From the
definition, $\E \Pi(B)=\mu(B)$, this is why the parameter measure
$\mu$ of a Poisson process is indeed the intensity measure of this
point process. To emphasise the dependence of the distribution on
$\mu$, we write the expectation as $\E_\mu$ in the sequel.

Consider a measurable function $G \colon \bN \to \R$ and define the difference operator
\begin{displaymath}
  D_{z}G(\phi) := G(\phi + \delta_{z}) - G(\phi),\ \phi\in \bN.
\end{displaymath}
For the iterations of the difference operator
\begin{displaymath}
  D_{z_{1}, \ldots, z_{n}}G = D_{z_{n}}(D_{z_{1}, \ldots, z_{n-1}}G),
\end{displaymath}
and every tuple of points
$(z_{1}, \ldots, z_{n})\in X^n$, it can be checked that 
\begin{displaymath}
   D_{z_{1}, \ldots, z_{n}}G(\nu) = \sum_{J \subset \{1, 2,
     \ldots, n\}} (-1)^{n - |J|} G \big(\nu + \Sigma_{j \in J}
   \delta_{z_{j}} \big), 
\end{displaymath}
where $|J|$ stands for the cardinality of $J$.  Define
\begin{displaymath}
  T_\mu G(z_{1}, \ldots, z_{n}) := \E_\mu D_{z_{1}, \ldots, z_{n}}G(\Pi). 
\end{displaymath} 
Suppose that the functional $G$ is such that there exists a constant
$c > 0$ satisfying
\begin{displaymath}
  |G \big({\Sigma}_{j= 1}^{n} \delta_{z_{j}} \big)| \le c^{n}
\end{displaymath}
for all $n \ge 1$ and all $(z_{1}, \ldots z_{n})$. It was proved
in~\cite[Theorem 2.1]{MolchanovI2000} that in the case of finite
measures $\mu, \mu'$ if then expectation $\E_{\mu+\mu'} G(\Pi)$ exists
then
\begin{equation} \label{MZequation}
  \E_{\mu+\mu'} G(\Pi) =\E_{\mu} G(\Pi)+ \sum_{n= 1}^{\infty}
  \frac{1}{n!} \int _{X^n}T_{\mu}G(z_{1},\dots, z_{n})
  \mu'(\mathrm{d}z_{1}) \ldots \mu'(\mathrm{d}z_{n}). 
\end{equation}
Generalisations of this formula to infinite and signed measures for
square integrable functionals can be found in \cite{Last:14}. A finite
order expansion formula can be obtained by representing the
expectation above in the form
$\E_{\mu+\mu'} G(\Pi)=\E_{\mu} \E_{\mu'}[G(\Pi+\Pi')\cond \Pi]$ where
$\Pi$ and $\Pi'$ are independent Poisson processes with intensity
measures $\mu$ and $\mu'$, respectively, and
then applying the moment expansion formula by
\cite[Theorem~3.1]{BlaMerSch:97} to $G(\Pi+\Pi')$ viewed as a
functional of $\Pi'$ with a given $\Pi$. This will give
\begin{multline} \label{fexp}
  \E_{\mu+\mu'} G(\Pi) =\E_{\mu} G(\Pi)+ \sum_{n= 1}^{k}
  \frac{1}{n!} \int _{X^n}T_{\mu}G(z_{1},\dots, z_{n})
  \mu'(\mathrm{d}z_{1}) \ldots \mu'(\mathrm{d}z_{n}) \\+
  \frac{1}{(k+1)!} \int _{X^{k+1}}T_{\mu+\mu'}G(z_{1},\dots, z_{k+1})
  \mu'(\mathrm{d}z_{1}) \ldots \mu'(\mathrm{d}z_{k+1}). 
\end{multline}

\section{Gradients of ChF loss function}\label{SM1}
The ChF method of estimating the compounding distribution $\Lambda$ or
more generally, the tripplet $(a,\sigma,\Lambda)$ of
the infinite divisible distribution, is based on fitting the empirical
characteristic function. The corresponding loss function $\lchf$ is
given by \eqref{eq:lchf}. It is everywhere differentiable in the usual
sense with respect to the parameters $a, \sigma$ and in Fr\'{e}chet
sense with respect to measure $\Lambda$.  Aiming at the steepest
descent gradient method for obtaining its minimum, we compute in this
section the gradients of $\lchf$ in terms of the following functions
\begin{align*}
  q_{1}(\theta, x) &:= \cos(\theta x) - 1, \quad Q_{1}(\theta,
                     \Lambda) :=   \int q_{1}(\theta, x) \Lambda(\rd
                     x); \\  
  q_{2}(\theta, x) &:= \sin(\theta x) - \theta x \one_{\{|x|<\epsilon\}},
                     \quad Q_{2}(\theta, a,\Lambda) :=  a \theta +
                     \int q_{2}(\theta, x) \Lambda(\rd x). 
\end{align*}
Using this notation, the real and imaginary parts of an infinitely
divisible distributions characteristic function $\phi=\phi_1+i\phi_2$
can be written down as
\begin{align*}
  \phi_1(\theta; a, \sigma, \Lambda)&= e^{h \{ Q_{1}(\theta, \Lambda)
                                      - \sigma^{2} \theta^{2} / 2 \}}
                                      \cos \{hQ_{2}(\theta,a, \Lambda)
                                      \}, \\ 
  \phi_2(\theta; a, \sigma, \Lambda) &= e^{h \{ Q_{1}(\theta, \Lambda)
                                       - \sigma^{2} \theta^{2} / 2 \}}
                                       \sin \{ hQ_{2}(\theta,
                                       a,\Lambda) \}. 
\end{align*}
After noticing that $\hat \phi_{n}=\hat \phi_{n,1}+i\hat \phi_{n,2}$, with
\begin{align*}
  \hat \phi_{n,1}(\theta)= \frac{1}{n} \sum_{j=1}^{n} \cos(\theta
  X_{j}), \qquad \hat \phi_{n,2}(\theta) = \frac{1}{n} \sum_{j=1}^{n}
  \sin(\theta X_{j}), 
\end{align*}
the loss functional $\lchf$ can be written as
\begin{equation*}
  \lchf(a, \sigma, \Lambda) = \int\big\{\phi_1(\theta; a, \sigma,
  \Lambda) - \hat \phi_{n,1}(\theta) \big\}^{2} \omega(\theta)\rd
  \theta 
  + \int\big\{\phi_2(\theta; a, \sigma, \Lambda) - \hat
  \phi_{n,2}(\theta) \big\}^{2} \omega(\theta)\rd \theta.  
\end{equation*}

From this representation, the following tree sets of formulae are
obtained in a straightforward way.
\begin{enumerate}
\item The partial derivative of the loss functional with respect to
  $a$ is equal to
\begin{align*}
  {\partial\over\partial a}\lchf (\theta; a, \sigma, \Lambda)
  &=  2 \int \{\phi_1(\theta; a, \sigma, \Lambda) -
    \hat\phi_{n,1}(\theta)\}{\partial\over\partial
    a}\phi_1(\theta; a, \sigma, \Lambda) 
    \omega(\theta) \rd \theta \\ 
  &\quad + 2 \int \{\phi_2(\theta; a, \sigma, \Lambda) -
    \hat\phi_{n,2}(\theta)\}{\partial\over\partial
    a}\phi_2(\theta; a, \sigma, \Lambda) 
    \omega(\theta)\rd \theta,
\end{align*}
where
\begin{align*}
  {\partial\over\partial a}
  \phi_1(\theta; a, \sigma, \Lambda)
  &  = - h\theta e^{h \{ Q_{1}(\theta, \Lambda) - \sigma^{2} \theta^{2} / 2 \}}                               \sin \{ hQ_{2}(\theta,a, \Lambda) \}, \\
  {\partial\over\partial a}\phi_2(\theta; a, \sigma, \Lambda)
  &= h\theta e^{h \{ Q_{1}(\theta, \Lambda) - \sigma^{2} \theta^{2} /
    2 \}} \cos \{ hQ_{2}(\theta, a,\Lambda) \}.
  \end{align*}

\item The partial derivative of the loss functional with respect to
  $\sigma$ is equal to
\begin{align*}
  {\partial\over\partial \sigma}\lchf (\theta; a, \sigma, \Lambda)
  &=  2 \int \{\phi_1(\theta; a, \sigma, \Lambda) - \hat\phi_{n,1}(\theta)\}
  {\partial\over\partial \sigma}\phi_1(\theta; a, \sigma, \Lambda) 
    \omega(\theta)\rd \theta \\ 
  &\quad + 2 \int \{\phi_2(\theta; a, \sigma, \Lambda) -
    \hat\phi_{n,2}(\theta)\} 
    {\partial\over\partial\sigma}\phi_2(\theta; a, \sigma, \Lambda)
    \omega(\theta)\rd \theta,
\end{align*}
where
\begin{align*}
    {\partial\over\partial \sigma}\phi_1(\theta; a, \sigma, \Lambda)
  &= - h\sigma \theta^{2} e^{h \{ Q_{1}(\theta, \Lambda) - \sigma^{2}
    \theta^{2} / 2 \}} \cos \{hQ_{2}(\theta, a,\Lambda) \}, \\
    {\partial\over\partial \sigma}\phi_2(\theta; a, \sigma, \Lambda)
  &= -h\sigma \theta^{2} e^{h \{ Q_{1}(\theta, \Lambda) - \sigma^{2}
    \theta^{2} / 2 \}} 
    \sin \{ hQ_{2}(\theta,a, \Lambda) \}.
  \end{align*}

\item Expression for the gradient function corresponding to the
  Fr\'echet derivative with respect to the measure $\Lambda$ is obtained using the
  Chain rule~\eqref{eq:chain}:
\begin{align*}
  \nabla \lchf(x; \Lambda) &=  2 \int\{\phi_1(\theta; a, \sigma,
                             \Lambda) -
                             \hat\phi_{n,1}(\theta)]\}\nabla\phi_1(\theta)[x,
                             \Lambda] \omega(\theta)\rd \theta \\ 
               & \quad + 2 \int \{\phi_2(\theta; a, \sigma, \Lambda) -
                 \hat\phi_{n,2}(\theta)\} \nabla\phi_2(\theta)[x,
                 \Lambda] \omega(\theta) \rd \theta, 
\end{align*}
where the gradients of
$\phi_i(\theta):=\phi_i(\theta; a, \sigma, \Lambda)$, $i=1,2$, with
respect to the measure $\Lambda$ are given by
\begin{align*}
  \nabla\phi_1(\theta)(x; \Lambda)  &= he^{h \{ Q_{1}(\theta, \Lambda)
                                     - \sigma^{2} \theta^{2} / 2 \}}
                                     \big\{ \cos \big(hQ_{2}(\theta,a,
                                     \Lambda)\big) q_{1}(\theta, x) -
                                     \sin\big(hQ_{2}(\theta,
                                     a,\Lambda)\big) q_{2}(\theta, x)
                                     \big\},  \\  
  \nabla\phi_2(\theta)(x; \Lambda)  &= he^{h \{ Q_{1}(\theta, \Lambda)
                                      - \sigma^{2} \theta^{2} / 2
                                      \}}\big\{ \sin
                                      \big(hQ_{2}(\theta,a,
                                      \Lambda)\big) q_{1}(\theta, x) +
                                      \cos\big(hQ_{2}(\theta,
                                      a,\Lambda)\big) q_{2}(\theta, x)
                                      \big\}.  
\end{align*}
\end{enumerate}

\section{Description of the CoF method}\label{SM3}
As it was alluded in the Introduction, the CoF method uses a representation
of the convolution as a function of the compounding measure. We now formulate the main
theoretical result of the paper on which the CoF method is based. 

\begin{theorem}\label{Main}
Let $W_t$ be a pure jump L\'evy process characterised by \eqref{La}
and $F(y)=F_h(y)$ be the cumulative distribution function of $W_h$. Then one has
  \begin{align}
    \begin{split}\label{eq:Main}
      F^{*2}(y) &= F(y)+\sum_{n=1}^{\infty} \frac{h^n}{n!}
      \int_{\R^n} U_{x_{1}, \ldots, x_{n}}F(y)
      \Lambda(\rd x_{1}) \ldots \Lambda(\rd x_{n}),
    \end{split}
  \end{align}
  where $U_{x_1}F(y)=F(y-x)-F(x)$ and
    \begin{equation}\label{U}
    U_{x_{1}, \ldots, x_{n}}F(y):=U_{x_{n}}(U_{x_1, \ldots,
      x_{n-1}}F(y)) = \sum_{J \subset \{1, 2, \ldots, n\}}
    (-1)^{i-|J|}F(y - \Sigma_{j \in J}x_{j}). 
  \end{equation}
 The sum above is taken over all the subsets $J$ of $\{1, 2, \ldots, n\}$
 including the empty set.
\end{theorem}

\begin{proof} 
  To prove the theorem, we use a coupling of $W_t$ with a Poisson
  process $\Pi$ on $\R_+\times \R$ driven by the intensity measure
  $\mu=\ell \times \Lambda$, where $\ell$ is the Lebesgue measure on
  $\R_+$.  For each realisation
  $\Pi = \Sigma_j \delta_{z_j}$ with $z_j=(t_{j}, x_{j})$, denote by $\Pi_t $ the
  restriction of $\Pi$ onto $[0, t] \times \R$.  Then, the L\'evy
  process can be represented as
  \begin{displaymath}
    W_{t} = \Sigma_{(t_j,x_{j} )\in \Pi_t} x_{j}=\int_0^t\int_{\R} x
    \Pi(\rd s\, \rd x).
  \end{displaymath}
  For a fixed arbitrary $y\in\R$ and a point configuration
  $\phi=\Sigma_j \delta_{(t_i,x_j)}$, consider a functional $G_y$
  defined by
  \begin{displaymath}
    G_y(\phi) = \one \Bigl\{\sum_{(t_j,x_{j} )\in \phi} x_j \le y\Bigr\}
  \end{displaymath}
  and notice that for any $z=(t,x)$,
  \begin{equation}\label{eq:addpoint}
    G_y(\phi+\delta_z)=\one \Bigl\{\sum_{(t_j,x_{j} )\in \phi} x_j \le y-x\Bigr\}=G_{y-x}(\phi).
  \end{equation}

  Clearly, the cumulative distribution function of $W_{h}$ can be
  expressed as an expectation 
  \begin{displaymath}
    F(y) = \P_\mu\Big\{\sum_{(t_j,x_{j} )\in \Pi_h} x_{j}\le y\Big\}=\E_\mu G_y(\Pi_h).
  \end{displaymath}
  Let $\mu'=[0,h]\times \Lambda$ and $\mu''=[h,2h]\times \Lambda$. Then
  \begin{displaymath}
    \E_{\mu'+\mu''} G_y(\Pi)=
    \P\{W_{2h} \le y\} = \P\{W_{h} + W_{h}'' \le y\} = F^{*2}(y),
  \end{displaymath}
  where $W_h''=W_{2h}-W_h$.
  Observe also that by iteration of \eqref{eq:addpoint},
  \begin{align*}
    T_{\mu'}G_y(z_{1}, \ldots, z_{n}) &= \E_{\mu'}
                                        D_{z_{1}, \ldots, z_{n}}G_y(\Pi)= \sum_{J \subset \{1, 2,
                                        \ldots, n\}} (-1)^{n - |J|} \E_{\mu'}
                                        G_y \big(\Pi + \Sigma_{j \in J}
                                        \delta_{z_{j}} \big)\\
                                      &= \sum_{J \subset \{1, 2,
                                        \ldots, n\}} (-1)^{n-|J|}F(y -
                                        \Sigma_{j \in
                                        J}x_{j})=U_{x_{1}, \ldots,
                                        x_{n}}F(y).  
  \end{align*}

  It remains now to apply expansion \eqref{MZequation} to complete the proof:
  \begin{align*} 
    F^{*2}(y)&= \E_{\mu'+\mu''} G_y(\Pi) \\
             & = F(y)+\sum_{n=1}^{\infty} \frac{1}{n!}
               \int_{(\R_+\times \R)^n} U_{x_{1}, \ldots, x_{n}}F(y)\,
               \mu''(\rd t_1\,\rd x_{1}) \ldots \mu''(\rd t_n\,\rd x_{n})\\
             &= F(y)+\sum_{n=1}^{\infty} \frac{h^n}{n!}
               \int_{\R^n} U_{x_{1}, \ldots, x_{n}}F(y)\,
               \Lambda(\rd x_{1}) \ldots \Lambda(\rd x_{n}).  
  \end{align*}
\end{proof}

The empirical convolution of a sample $(X_1\thru X_n)$,
\begin{equation}\label{eq:2FoldEstim}
  \hat F_{n}^{*2}(y) := \frac{1}{\binom{n}{2}} \sum_{1\le i<j\le n}
  \one\{X_i + X_j \le y \}.
\end{equation}
is an unbiased and consistent estimator of $F^{*2}(x)$,
see~\cite{Frees1986}.

The CoF-method looks for a finite measure $\Lambda$ that minimises the
following loss function
  \begin{equation}\label{eq:lcof}
    \lcof^{(k)}(\Lambda) = \int\Big\{ 
    \hat F(y)+\sum_{i = 1}^{k} \frac{h^{i}}{i!} \int_{\mathbb{R}^{i}}
    U_{x_{1}, \ldots, x_{i}} \hat F_{n}(y) \Lambda(\rd x_{1}) \ldots
    \Lambda(\rd x_{i}) -\hat F_{n}^{*2}(y)\Big\}^{2} \omega(y)\rd y.
\end{equation}

The infinite sum in~\eqref{eq:Main} is truncated to $k$ terms
in~\eqref{eq:lcof} for computational reasons. The error introduced by
the truncation can be accurately estimated by bounding the remainder term in the
finite expansion formula~\eqref{fexp}. Alternatively, turning
to~\eqref{eq:Main} and using $0\leq F(y)\leq 1$, 
we obtain that
$|U_{x_{1}, \ldots, x_{i}}F(y)|\leq 2^{i-1}$ for all $i=1,2,\dots$,
yielding
\begin{displaymath}
  \Big|\frac{h^{k+1}}{(k+1)!}
  \int_{\R^{k+1}} U_{x_{1}, \ldots, x_{k+1}}F(y) \Lambda(\rd x_{1}) \ldots
  \Lambda(\rd x_{k+1})\Big|\leq
  \frac{1}{2}\sum_{i=k+1}^\infty\frac{(2h\|\Lambda\|)^{i}}{i!}.
\end{displaymath}
Notice that the upper bound corresponds to a half the distribution
tail $\P\{Z\geq k+1\}$ of a Poisson random variable, say
$Z\sim \Po(2h\|\Lambda\|)$.  Thus, to have a good estimate with this
method, one should either calibrate the time step $h$ (if the data are
coming from the discretisation of a L\'evy process trajectory) or to
use higher $k$ to make the remainder term small enough. For instance,
for the horse kick data considered in Introduction~\ref{sec:intro},
$h=1$ and $\|\Lambda\|=0.61$. The resulting error bounds for $k=1,2,3$
are 0.172, 0.062 and 0.017, respectively, which shows that $k=3$ is
rather adequate cutoff for this data. Since $h\|\Lambda\|$ is the mean
number of jumps in the strip $[0,h]\times\R$, in practice one should
aim to choose $h$ so that to have only a few jumps with high
probability. If, on the contrary, the number of jumps is high, their
sum by the Central Limit theorem would be close to the limiting law
which, in the case of a finite variance of jumps, is Normal and so
depends on the first two moments only and not on the entire compounding
distribution. Therefore an effective estimation of
$\Lambda/\|\Lambda\|$ is impossible in this case, see \cite{duval2014}
for a related discussion.

As with the ChF method, the CoF algorithm relies on the steepest descent
approach. The needed gradient function has the form
  \begin{align*}
    \nabla &\lcof^{(k)} (x; \Lambda) \\
    &= 2h \int \Big\{ \hat F_n(y) + \sum_{i = 1}^{k} \frac{h^{i}}{i!}
      \int_{\mathbb{R}^{i}} U_{x_{1}, \ldots, x_{i}} \hat F_{n}(y)
      \Lambda(\rd x_{1}) \ldots \Lambda(\rd x_{i}) -\hat
      F_{n}^{*2}(y)\Big\} \Gamma(y,x,\Lambda) \omega(y)\rd y, 
       \end{align*}
       where
    \begin{align*}
 \Gamma(y,x,\Lambda)=\sum_{j= 0}^{k-1} \frac{h^{j}}{j!}
      \int_{\mathbb{R}^{j}} (U_{x, x_{1}, \ldots, x_{j}} \hat
      F_{n})(y) \Lambda(\rd x_{1}) \ldots \Lambda(\rd x_{j}).
          \end{align*}
 
      This formula follows from the Chain rule~\eqref{eq:chain} and the equality
   \begin{align*}
 \nabla \Big(\sum_{j= 1}^{k} \frac{h^{j}}{j!}
      \int_{\mathbb{R}^{j}} &(U_{x_{1}, \ldots, x_{j}} \hat
      F_{n})(y) \Lambda(\rd x_{1}) \ldots \Lambda(\rd x_{j}) \Big) (x; \Lambda)\\
      &= h\sum_{j= 0}^{k-1} \frac{h^{j}}{j!}
      \int_{\mathbb{R}^{j}} (U_{x, x_{1}, \ldots, x_{j}} \hat
      F_{n})(y) \Lambda(\rd x_{1}) \ldots \Lambda(\rd x_{j}).
          \end{align*}
          To justify the last identity, it suffices to see that for
          any integrable symmetric function $u(x_{1}, \ldots,x_{j}) $
          of $j\ge1$ variables,
 \begin{align*}
 \nabla \Big(
      \int_{\mathbb{R}^{j}} u(x_{1}, \ldots,x_{j}) &\Lambda(\rd x_{1})
                                                     \ldots\Lambda(\rd
                                                     x_{j}) \Big) (x;
                                                     \Lambda)\\ 
      &= j\int _{\mathbb{R}^{j-1}}u(x,x_{1},\ldots,x_{j-1})
        \Lambda(\rd x_{1}) \ldots \Lambda(\rd x_{j- 1}), 
          \end{align*}
which holds due to 
 \begin{align*}
      \int_{\mathbb{R}^{j}} &u(x_{1}, \ldots,x_{j}) (\Lambda+\nu)(\rd
                              x_{1}) \ldots(\Lambda+\nu)(\rd x_{j}) - 
      \int_{\mathbb{R}^{j}} u(x_{1}, \ldots,x_{j}) \Lambda(\rd x_{1})
                              \ldots\Lambda(\rd x_{j}) \\ 
      &= \sum_{k=1}^j\int _{\mathbb{R}^{j}}u(x_{1},\ldots,x_{j})
        \Lambda(\rd x_{1}) \ldots \Lambda(\rd x_{k- 1})\nu(\rd
        x_k)\Lambda(\rd x_{k+1})\ldots \Lambda(\rd x_{j})+o(\|\nu\|)\\ 
      &= j\int _{\mathbb{R}^{j}}u(x,x_{1},\ldots,x_{j-1})\nu(\rd x)
        \Lambda(\rd x_{1}) \ldots \Lambda(\rd x_{j- 1})+o(\|\nu\|). 
          \end{align*}
 
\section{Algorithmic implementation of the steepest descent method}
\label{sec:implement}
In this section we describe the algorithm implementing the gradient
descent method, which was used to obtain our simulation results
presented in Section~\ref{Ssd}.

Recall that the principal optimisation problem has the form
\eqref{eq:1}, where the functional $L(\Lambda)$ is minimised over
$\Lambda \in \mathbb{M}_{+}$ subject to the constraints on being a
L\'evy measure. For computational purposes the measure
$\Lambda \in \mathbb{M}_{+}$ is replaced by its discrete approximation
which has a form of a linear combination
$\boldsymbol{\Lambda}=\Sigma_{i=1}^{l} \lambda_{i} \delta_{x_{i}}$ of
Dirac measures on a finite regular grid $x_{1}, \ldots, x_{l}\in\R$,
$x_{i + 1} =x_{i}+2\Delta$. Specifically, for a given measure
$\Lambda$, the atoms of $\boldsymbol{\Lambda}$ are given by
\begin{align}
  \lambda_{1} &:= \Lambda((-\infty, x_{1} + \Delta)), \notag\\
  \lambda_{i} &:= \Lambda( [x_{i} - \Delta, x_{i} + \Delta )), \quad
                \text{for } i = 2, \ldots, l - 1, \label{eq:discr}\\ 
  \lambda_{l} &:= \Lambda([x_{l} - \Delta, \infty)).  \notag
\end{align}
Clearly, the larger is $l$ and the finer is the grid
$\{x_{1}, \ldots, x_{l}\}$ the better is approximation, however, at a
higher computational cost.

Respectively, the discretised version of the gradient function $\nabla
L(x;\Lambda)$ is the vector
\begin{displaymath}
  \boldsymbol{g}=(g_1,\ldots, g_l),\quad g_{i} := \nabla L
  (x_i;\boldsymbol{\Lambda}), \quad i= 1, \ldots, l. 
\end{displaymath}
For example, the cost function
$L=L^{(1)}_{\rm CoF}$ with $\omega(y)\equiv 1$ has the gradient
\begin{align*}
  \nabla &\lcof^{(1)} (x; \Lambda) = 2h \int \Big\{ \hat F_n(y)  -\hat
           F_{n}^{*2}(y)+   \int  \hat F_{n}(y-z)    \Lambda(\rd z)
           \Big\} \hat F_{n}(y-x)\,\rd y. 
\end{align*}
The discretised gradient for this
example is the vector  $\boldsymbol{g}$ with the components
\begin{equation}
 \label{exa}
 g_i=2h \int \Big\{ \hat F_n(y)  -\hat  F_{n}^{*2}(y)+   \sum_{j=1}^l
 \hat F_{n}(y-x_j) \lambda_j \Big\} \hat F_{n}(y-x_i)\,\rd y,\quad
 i=1\thru l.
\end{equation}

Our main optimisation algorithm has the following structure:
\begin{algorithm}[H]
\caption{Steepest descent algorithm}
\begin{algorithmic}[1]
\Require initial vector $\boldsymbol{\Lambda}$
\Function{GoSteep}{$\boldsymbol{\Lambda}$}
\State initialise the discretised gradient $\boldsymbol{g} \gets (\nabla
L(x_{1};\boldsymbol{\Lambda}), \ldots, \nabla
L(x_{l};\boldsymbol{\Lambda}))$ 
\While{\big( $\min_{i} g_{i} < - \tau_2$ \textbf{or} $\max_{\{i \colon
    \lambda_{i} > \tau_1\}}g_{i} > \tau_2$ \big)}
\State choose a favourable step size $\epsilon$ depending on $L$ and
$\boldsymbol{\Lambda}$ \State compute new vector
$\boldsymbol{\Lambda} \gets$ \Call{MakeStep}{$\epsilon$,
  $\boldsymbol{\Lambda}$, $\boldsymbol{g}$} \State compute gradient at
the new $\boldsymbol{\Lambda}$:
$\boldsymbol{g} \gets (\nabla L(x_{1};\boldsymbol{\Lambda}), \ldots,
\nabla L(x_{l};\boldsymbol{\Lambda}))$
\EndWhile
\State \Return{$\boldsymbol{\Lambda}$}
\EndFunction
\end{algorithmic}
\end{algorithm}
In the master algorithm description above, the line 3 uses the
necessary condition~\eqref{eq:KT} as a test condition for the main
cycle.  In the computer realisations we usually want to discard the
atoms of a negligible size: for this purpose we use a zero-value
threshold parameter $\tau_1$. We use another threshold parameter
$\tau_2$ to decide when the coordinates of the gradient vector are
sufficiently small. For the examples considered in the next section,
we typically used the following values: $\omega\equiv 1$, $\tau_1=10^{-2}$ and
$\tau_2=10^{-6}$. The key {\sc MakeStep} subroutine, mentioned on line
5, is described below. It calculates the admissible steepest direction
$\boldsymbol{\nu}^{*}$ of size $\|\boldsymbol{\nu}^{*}\|\leq \epsilon$
and returns an updated vector
$\boldsymbol{\Lambda} \leftarrow \boldsymbol{\Lambda} +
\boldsymbol{\nu}^{*}$.

\begin{algorithm}[H]
\caption{Algorithm for a steepest descent move 
}
\begin{algorithmic}[1]
\Require maximal step size $\epsilon$, current variable value $\boldsymbol{\Lambda}$ and current gradient value $\boldsymbol{g}$
\Function{MakeStep}{$\epsilon$, $\boldsymbol{\Lambda}$, $\boldsymbol{g}$}
\State initialise the optimal step $\boldsymbol{\nu^{*}} \gets \boldsymbol{0}$
\State initialise the running coordinate $i \gets 0$
\State initialise  the total mass available  $E \gets \epsilon$
\While{ (($E > 0$) \textbf{and} ($i \le l$)) }
\If{ $g_{i} > |g_{l}|$ }
\State $\nu_{i}^{*} \gets \max(-\lambda_{i}, -E)$
\State $E \gets E  - \nu_{i}^{*}$
\Else
\State $\nu_{l}^{*} \gets E$
\State $E \gets 0$
\EndIf
\State $i \gets i + 1$
\EndWhile
\State \Return{ $\boldsymbol{\Lambda} + \boldsymbol{\nu^{*}}$ }
\EndFunction
\end{algorithmic}
\end{algorithm}

The {\sc MakeStep} subroutine looks for a vector $\boldsymbol{\nu}^*$
which minimises the linear form $\sum_{i=1}^{l} g_{i} \nu_{i}$
appearing in the Taylor expansion
\begin{displaymath}
  L(\boldsymbol\Lambda+\boldsymbol\nu)-L(\boldsymbol\Lambda)=\sum_{i=1}^{l}
  g_{i} \nu_{i}+o(|\boldsymbol\nu|).   
\end{displaymath}
This minimisation is subject to the following linear constraints
\begin{align*}
 &\sum_{i = 1}^{l} |\nu_{i}| \le \epsilon,\\
 &\nu_i\ge -\lambda_i,\quad i=1,\ldots, l.
\end{align*}
The just described linear programming task has a straightforward
solution given below.

For simplicity we assume that $g_{1} \ge \ldots \ge g_{l}$.  Note that
this ordering can always be achieved by a permutation of the
components of the vector $\boldsymbol{g}$ and respectively, $\boldsymbol{\Lambda}$.
Assume also that the total mass of $\boldsymbol{\Lambda}$ is bigger than
the stepsize $\epsilon$. Define two indices
\begin{equation*}
  i_g = \max\{i \colon g_{i} \geq |g_{l}|\}, \quad \text{and} \quad
  i_{\epsilon} = \max\{i \colon \sum_{j = 1}^{i - 1} \lambda_{j} <
  \epsilon\}, \quad \epsilon > 0. 
\end{equation*}
If $i_{\epsilon} \le i_g$, then the coordinates of
$\boldsymbol{\nu}^{*}$ are given by
\begin{displaymath}
  \nu_{i}^{*} := 
\left\{
\begin{array}{ll}
 -\lambda_{i}& \text{for } i \le i_{\epsilon},    \\
\sum_{j = 1}^{i_{\epsilon} - 1} \lambda_{j} - \epsilon  & \text{for }  i = i_{\epsilon} + 1,     \\
0  & \text{for }  i \ge i_{\epsilon} + 2,   
\end{array}
\right.
\end{displaymath}
and if $i_{\epsilon} > i_g$,  then 
\begin{displaymath}
  \nu_{i}^{*} := 
\left\{
\begin{array}{ll}
 -\lambda_{i}, &  \text{for }  i \le i_{g},    \\
0  & \text{for }  i_g < i < l,     \\
 \epsilon - \sum_{j = 1}^{i_g} \lambda_{j}  & \text{for }  i = l.
\end{array}
\right.
\end{displaymath}

The presented algorithm is realised in the statistical computation environment
\texttt{R} (see \cite{R}) in the form of a library
\texttt{mesop} which is freely downloadable from one of the authors'
webpage.\footnote{\url{http://www.math.chalmers.se/~sergei/download.html}}

\section{Simulation results for a discrete L\'evy measure}\label{Ssd}

To illustrate the performance of our estimation methods we generated
samples of size $n=1000$ for compound Poisson processes driven by
different kinds of L\'evy measure $\Lambda$. For all examples in this
section, we implement three versions of the CoF with
$h=1$, $k=1,2,3$ and $\omega\equiv 1$. We also apply ChF using the estimate of CoF with
$k=1$. Observe that CoF with $k=1$ can be made particularly fast
because here we have a non-negative least squares optimisation
problem. 

\paragraph{ Poisson compounding distribution.}
Here we consider the simplest possible L\'evy measure
$\Lambda(dx)=\delta_1(dx)$ which corresponds to a standard Poisson
process with parameter 1. Since all the jumps are integer valued and
non-negative, it is logical take the non-negative integer grid for
possible atom positions of the discretised $\Lambda$. This is the
way we have done it for the horse kick data analysis. However, to test
the robustness of our methods, we took the grid
$\{0,\pm 1/4,\pm 2/4,\pm 3/4,\ldots\}$. As a result the estimated
measures might place some mass on non-integer points or even on
negative values of $x$ to compensate for inaccurately fitted
positive jumps. We have chosen to show on the graphs the discrepancies
between the estimated and the true measure. An important indicator of
the effectiveness of an estimation is the closeness of the total masses
$\|\hat\Lambda\|$ and $\|\Lambda\|$. For $\Lambda=\delta_1$, the probability to have
more than 3 jumps is approximately 0.02, therefore we expect that
$k=3$ would give an adequate estimate for this data. Indeed, the left
panel of Figure~\ref{fPois} demonstrates that the CoF with $k=3$ is
quite effective in detecting the jumps of the Poisson process compared
to $k=2$ and especially to $k=1$ which generate large discrepancies
both in atom sizes and in the total mass of the obtained measure. Observe
also the presence of artefact small atoms at large $x$ and even at
some non-integer locations. 

The right panel shows that a good alternative to a rather
computationally demanding CoF method with $k=3$, is a much faster
combined CoF--ChF method when $\hat{\Lambda}_1$ measure is used as the
initial measure in the ChF algorithm. The resulting measure
$\dhat{\Lambda}_1$ is almost idetical to $\hat{\Lambda}_3$, but also
has the total mass closer to the target value 1. The total variation
distances between the estimated measure and the theoretical one are
0.435, 0.084 and 0.053  for $k=1,2,3$, respectively. For the combined method it is
0.043 which is the best approximation in the total variation to the
original measure.

\begin{figure}[H]
\begin{subfigure}{.5\textwidth}
  \centering
  \includegraphics[width=.95\linewidth]{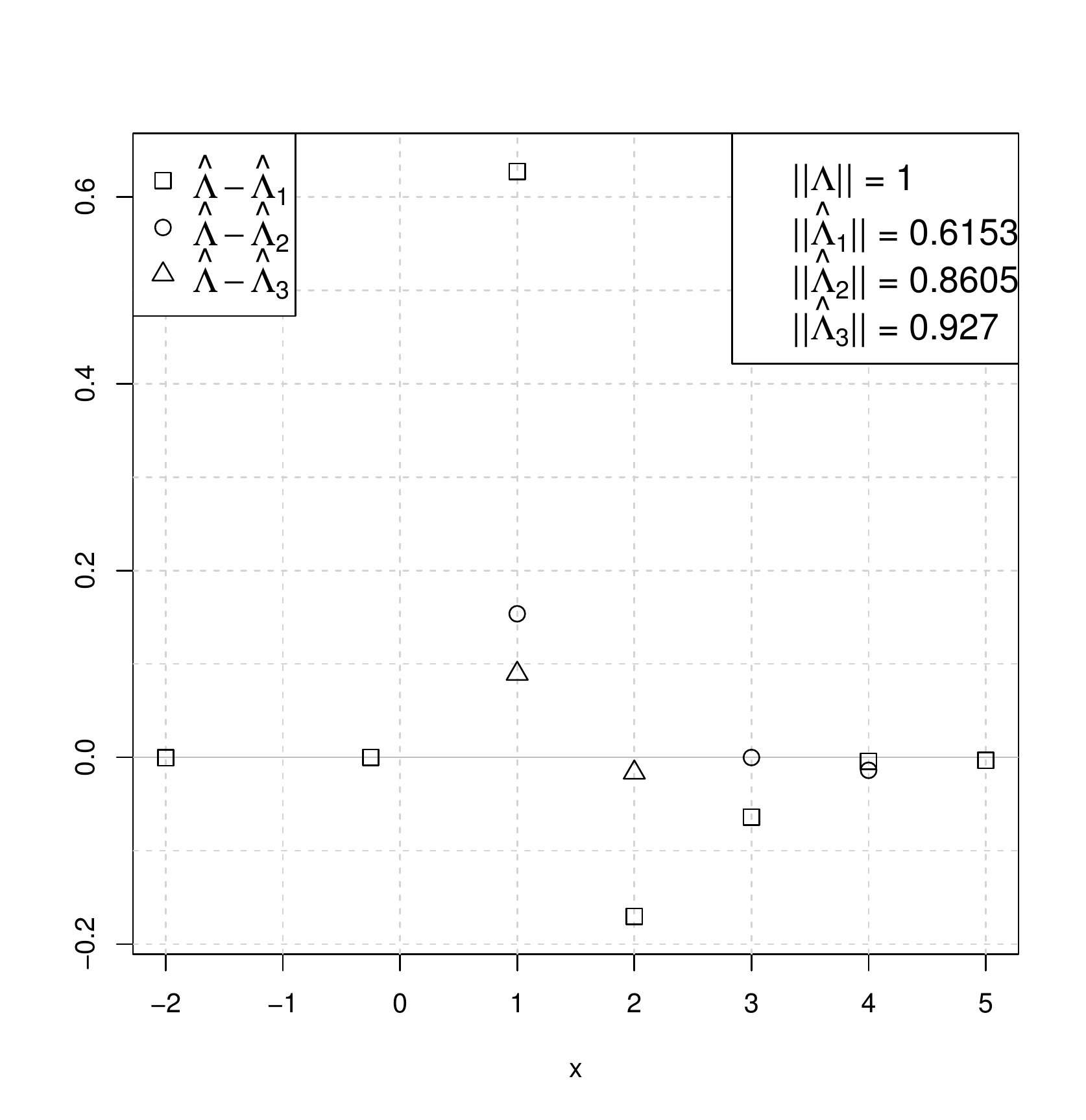}
\end{subfigure}%
\begin{subfigure}{.5\textwidth}
  \centering
  \includegraphics[width=.95\linewidth]{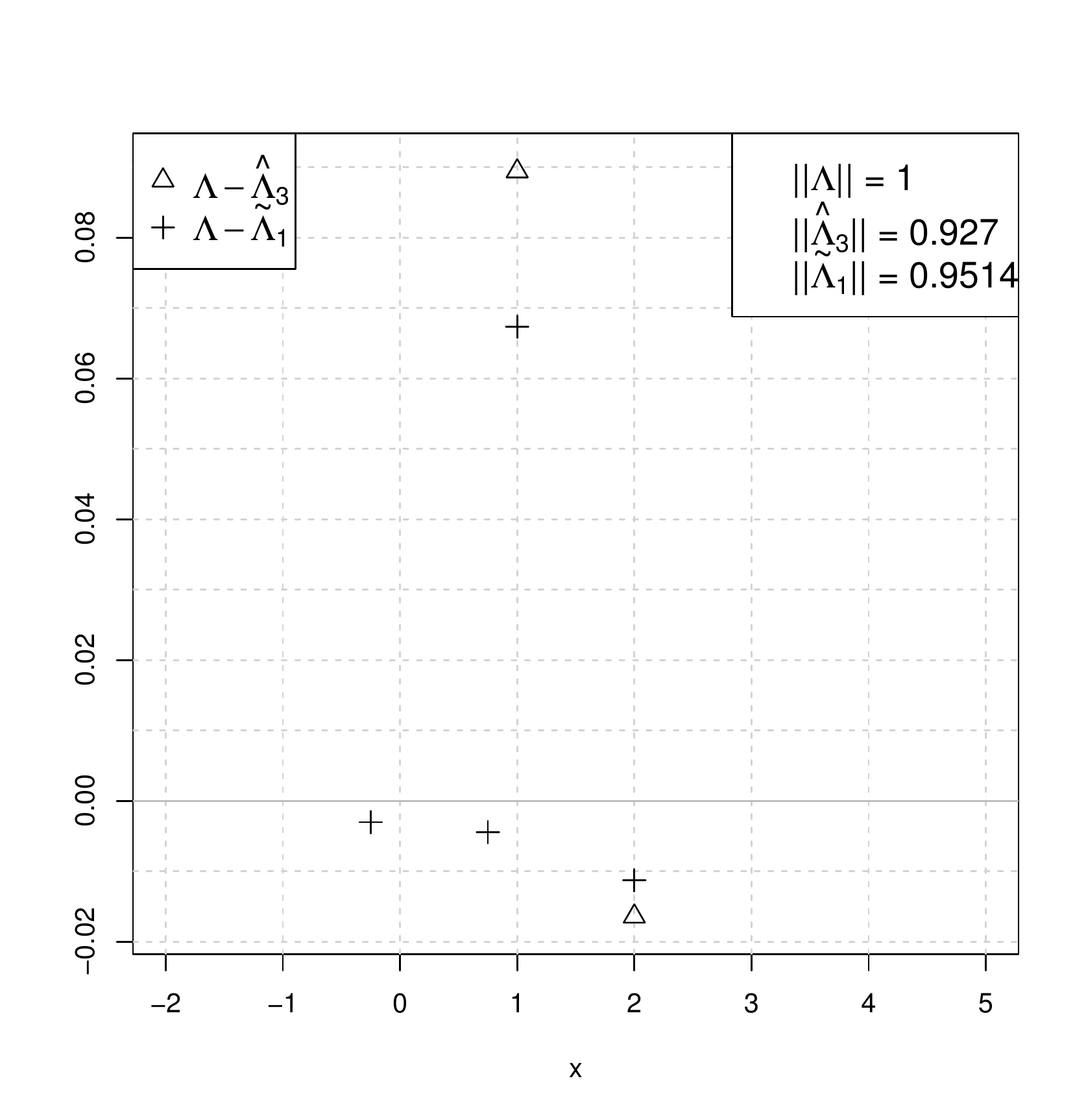}
\end{subfigure}
\caption{Simulation results for a Poisson $\Po(1)$ compounding distribution
  corresponding to $\Lambda$ concentrated at point $1$,
  with total mass $1$.  Left panel: the differences between
  $\Lambda(\{x\})$ and their estimates $\hat\Lambda_k(\{x\})$ obtained by
  CoF with $k=1, 2, 3$. Zero values of the differences are not
  plotted. Right panel: comparison of $\hat\Lambda_3$ with
  $\dhat{\Lambda}_1$ obtained by ChF initiated at
  $\hat\Lambda_1$. Notice the drastic change in the vertical axis
  scale as we go from the left to the right panel.}
\label{fPois}
\end{figure}

\paragraph{Compounding distribution with positive and negative jumps.}
Figure \ref{fPois1} presents the results on a compound Poisson process
with jumps of sizes $-1,1,2$ having respective probabilities 0.2, 0.2
and 0.6, so that $\Lambda=0.2\delta_{-1}+0.2\delta_1+0.6\delta_2$. The
overall intensity of the jumps is again $\|\Lambda\|=1$. The presence
of negative jumps canceling positive jumps creates an additional
difficulty for the estimation task. This phenomenon explains why the
approximation obtained with $k=2$ is worse than with $k=1$ and $k=3$:
two jumps of sizes +1 and -1 sometimes cancel each other, which is
indistinguishable from no jumps case.  Moreover, -1 and
2 added together is the same as having a single size 1 jump. The
left panel confirms that going from $k=1$ through $k=2$ up to $k=3$
improves the performance of CoF although the computing time increases
drastically. The corresponding total variation distances of
$\hat{\Lambda}_k$ to the theoretical distribution are 0.3669, 0.6268
and 0.1558. The combined method gives the distance 0.0975 and
according to the right plot is again a clear winner in this case
too. It is also much faster.

\begin{figure}[H]
\begin{subfigure}{.5\textwidth}
  \centering
  \includegraphics[width=.95\linewidth]{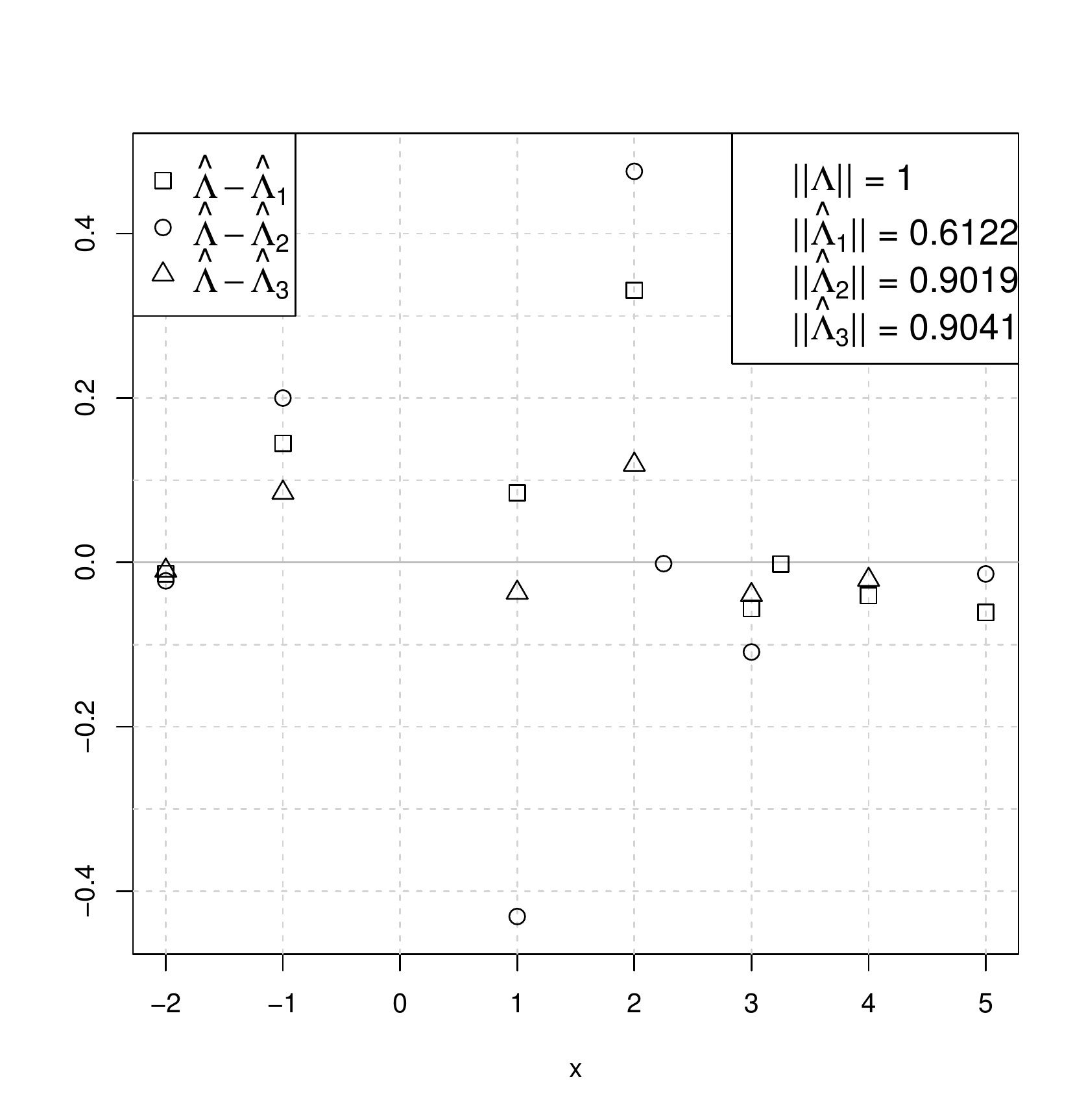}
\end{subfigure}%
\begin{subfigure}{.5\textwidth}
  \centering
  \includegraphics[width=.95\linewidth]{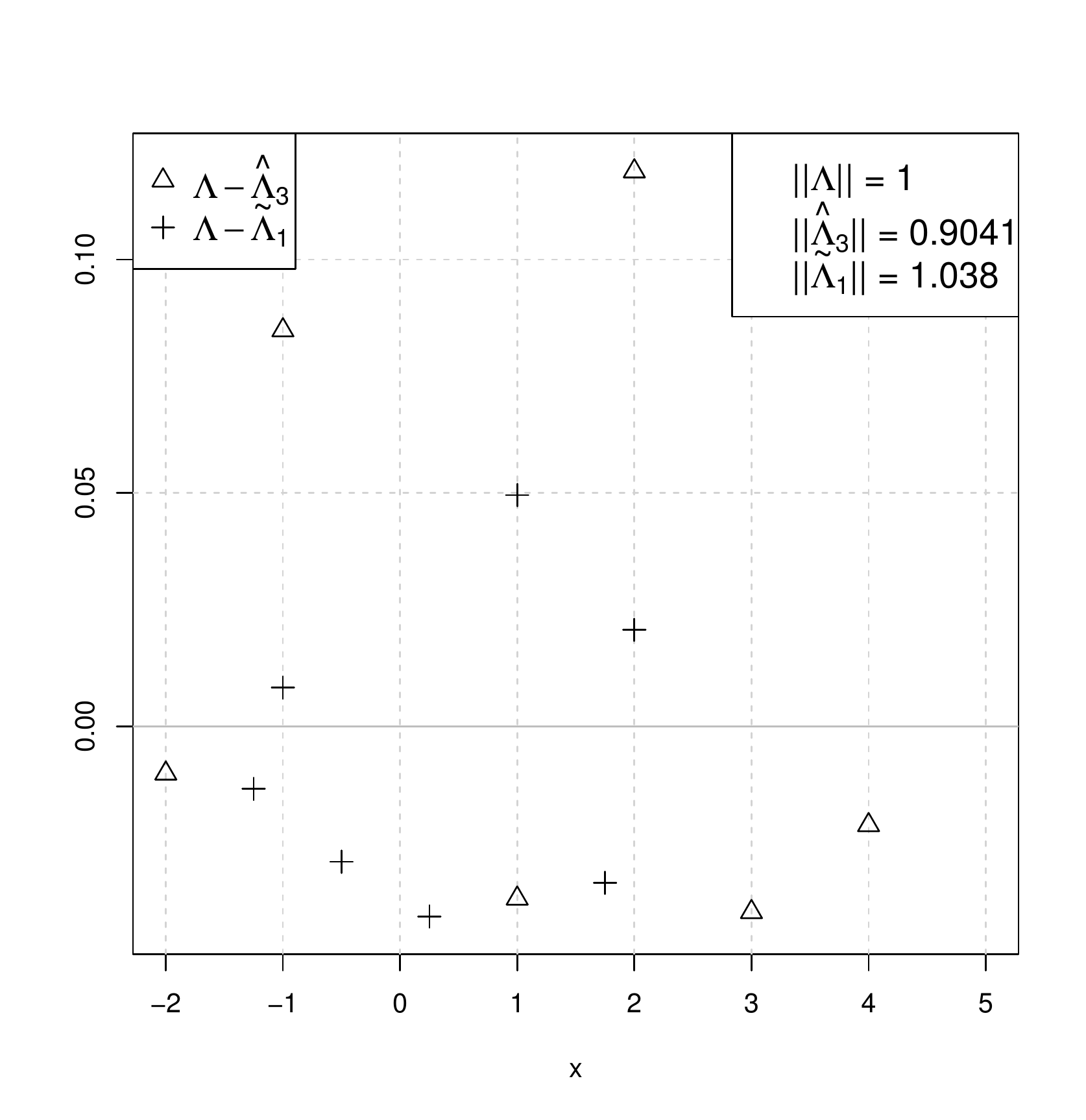}
\end{subfigure}
\caption{Simulation results for
  $\Lambda=0.2\delta_{-1}+0.2\delta_1+0.6\delta_2$.  Left panel: the
  differences between $\Lambda(\{x\})$ and their estimates
  $\hat\Lambda_k(\{x\})$ obtained by CoF with $k=1, 2, 3$. Right panel:
  comparison of $\hat\Lambda_3$ with $\dhat{\Lambda}_1$ 
}
\label{fPois1}
\end{figure}

\begin{figure}[H]
\begin{subfigure}{.5\textwidth}
  \centering
  \includegraphics[width=.95\linewidth]{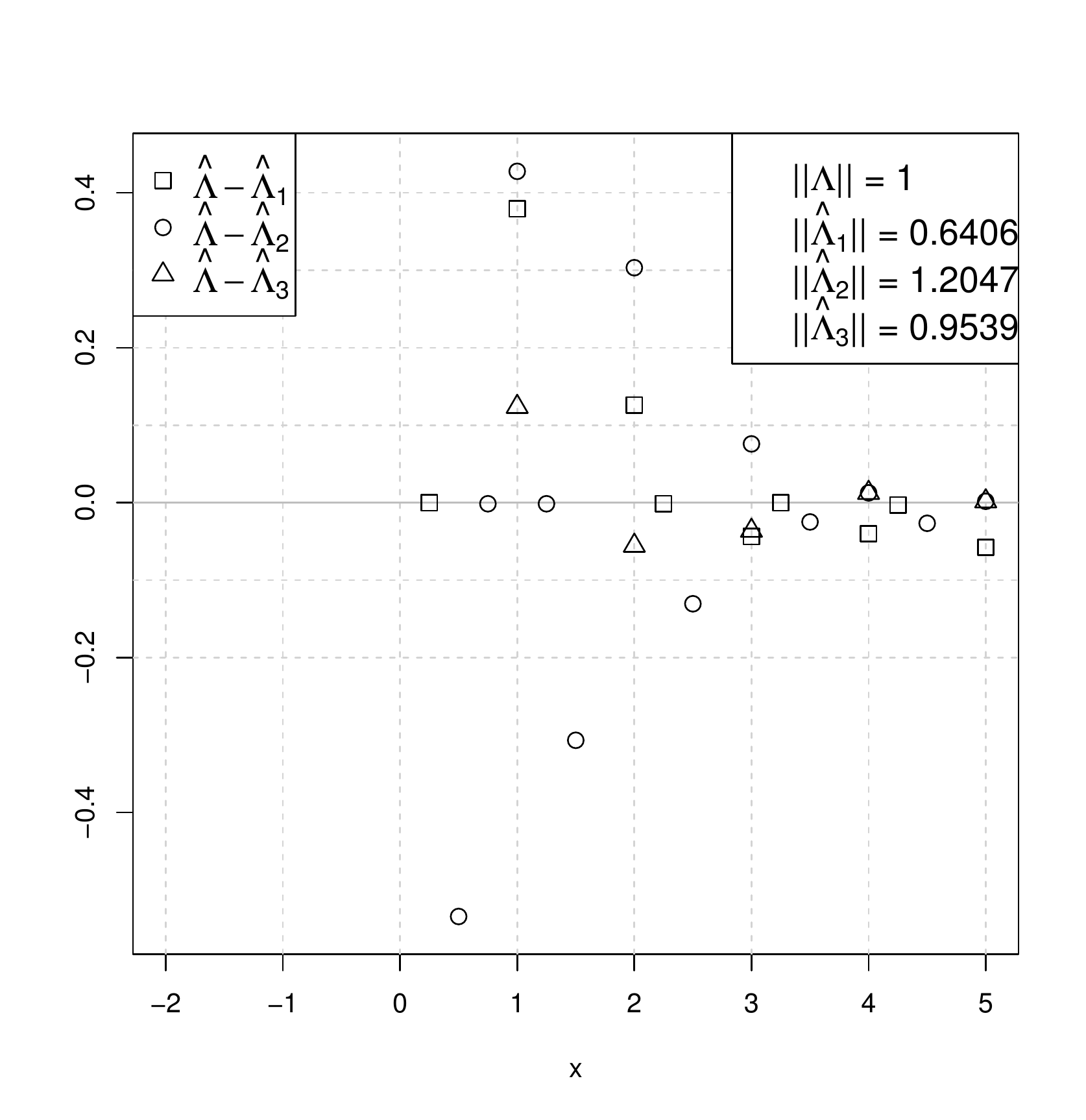}
\end{subfigure}%
\begin{subfigure}{.5\textwidth}
  \centering
  \includegraphics[width=.95\linewidth]{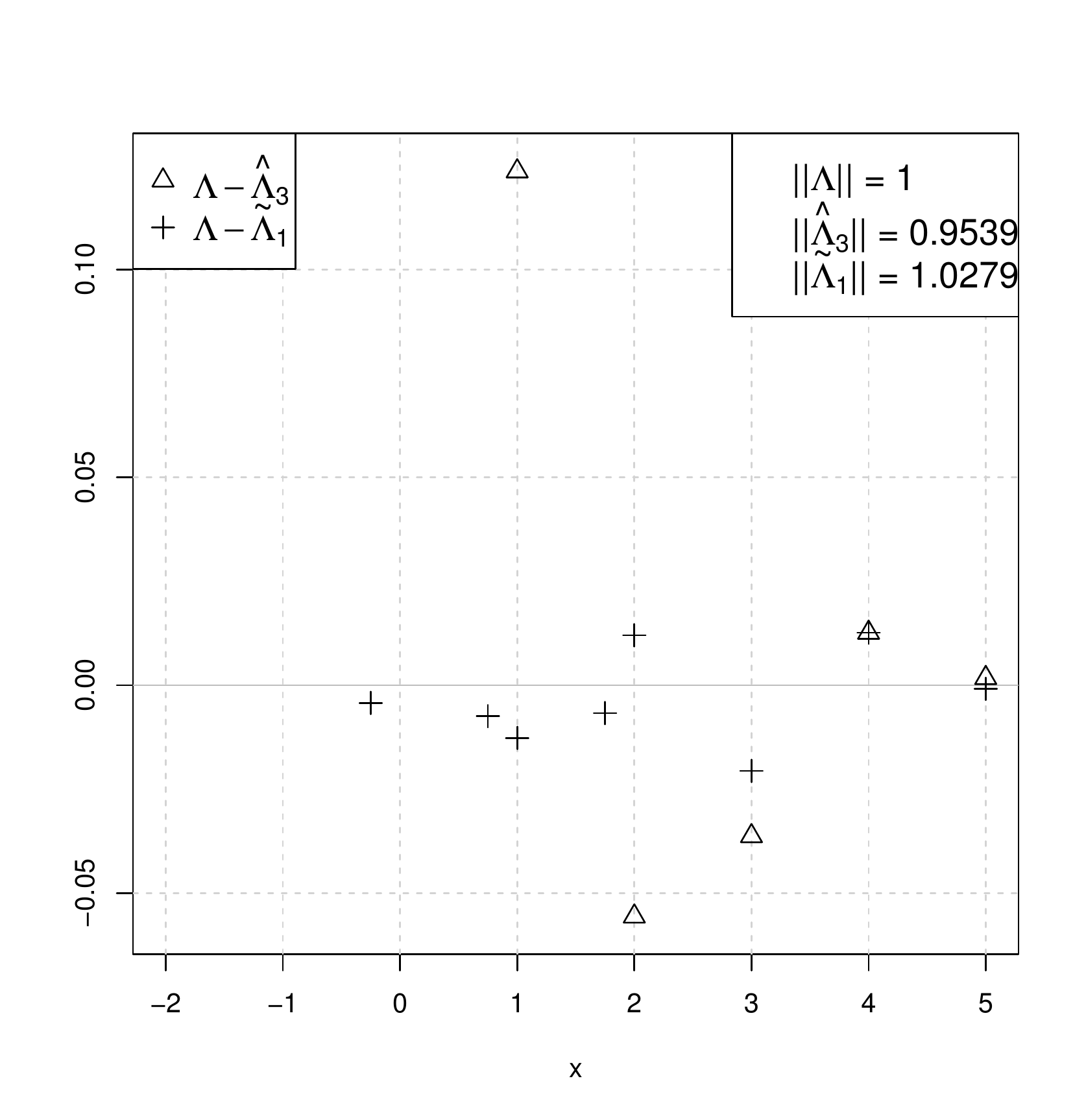}
\end{subfigure}
\caption{Simulation results for a shifted Poisson distribution
  $\Lambda(\{x\})={e^{-1}/ (x-1)!}$ for $x=1, 2, \ldots$.  Left panel:
  the differences between $\Lambda(\{x\})$ and their estimates
  $\hat\Lambda_k(\{x\})$ obtained by CoF with $k=1, 2, 3$.  Right panel:
  comparison of $\hat\Lambda_3$ with $\dhat{\Lambda}_1$ obtained by
  ChF initiated at $\hat\Lambda_1$.}
 \label{poi}
\end{figure}

\paragraph{Unbounded compounding distribution.}
 On Figure \ref{poi} we present the simulation results for
a discrete measure $\Lambda$ having an infinite support $\N$.  For
the computation, we limit the support range for the measures
in question to the interval $x\in[-2,5]$.  As the left panel reveals,
also in this case the CoF method with $k=3$ gives a better approximation than
$k=1$ or $k=2$ (the total variation distances to the theoretical
distribution is 0.1150 compared to 0.3256 and  0.9235, respectively)
and the combined faster method gives even better estimate with
$\dtv(\dhat{\Lambda}_1,\Lambda)=0.0386$. Interestingly, $k=2$ was the
worst in terms of the total variation distance. We suspect that the
'pairing effect' may be responsible: the jumps are better fitted with
a single integer valued variable rather than the sum of two. The
algorithm may also got stuck in a local minimum producing small atoms
at non-integer positions.

Finally, we present simulation results for two
cases of continuously distributed jumps. 
The continuous measures are replaced by their discretised versions given
by~\eqref{eq:discr}. In the examples below the grid size is $\Delta=0.25$.

\paragraph{Continuous non-negative compounding distribution.}
Figure \ref{fig:exp} summarises our simulation results for the
compound Poisson distribution with the jump sizes following the
exponential $\Exp(1)$ distribution.  The left plot shows that also in
this case the accuracy of approximation increases with $k$. Observe
that the total variation distance
$\dtv(\hat{\Lambda}_3,\boldsymbol{\Lambda})= 0.0985$ is
comparable with the discretisation error:
$\dtv(\Lambda,\boldsymbol{\Lambda})=0.075$. A Gaussian kernel smoothed
version of $\hat{\Lambda}_3$ is presented at the right plot of Figure~\ref{fig:exp}. The
visible discrepancy for small values of $x$ is explained by the fact
that there were no sufficient number of really small jumps in the
simulated sample to give the algorithm sufficient grounds to put more
mass around 0. Interestingly, the combined algorithm produced a
measure with a smaller (compared to $\hat{\Lambda}_3$) value of the $\lchf$,
but a larger total variation distance from $\boldsymbol{\Lambda}$.
Optimisation in the space of measures usually tends to produce atomic
measures since these are boundary points of the typical constraint
sets in $\M$. Indeed, $\dhat{\Lambda}_1$ has smaller number of atoms than
$\boldsymbol{\Lambda}$ does and still it better approximates the
empirical characteristic function
of the sample. It shows that the case of optimisation in the class of
absolutely continuous measures should be analysed differently by
characterising their tangent cones and deriving the
corresponding steepest descent methods. Additional conditions on the
density must also be imposed, like Lipschitz kind of conditions, to
make the feasible set closed in the corresponding measure topology.
\begin{figure}[H]
\begin{subfigure}{.5\textwidth}
  \centering
  \includegraphics[width=.95\linewidth]{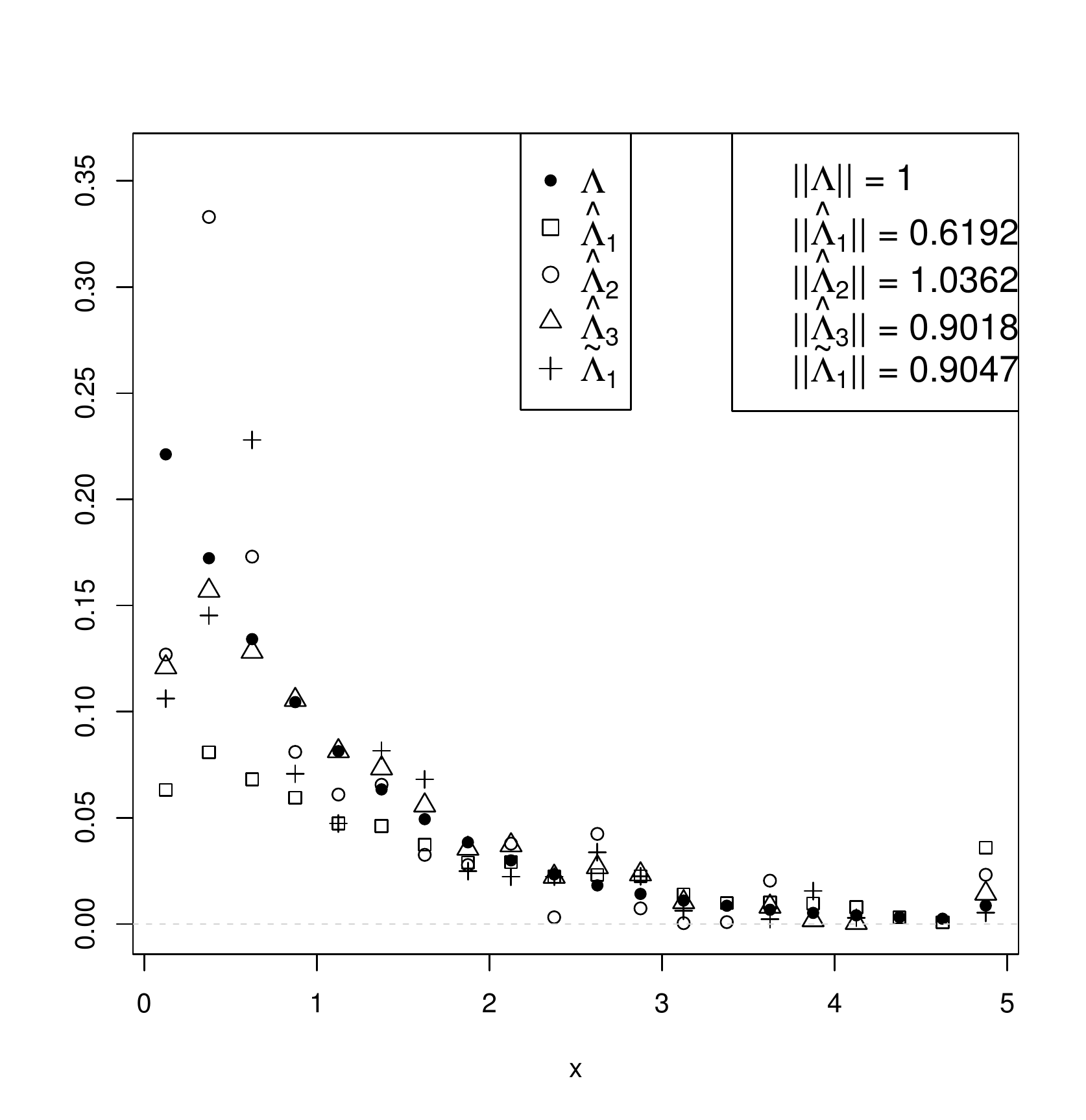}
\end{subfigure}%
\begin{subfigure}{.5\textwidth}
  \centering
  \includegraphics[width=.95\linewidth]{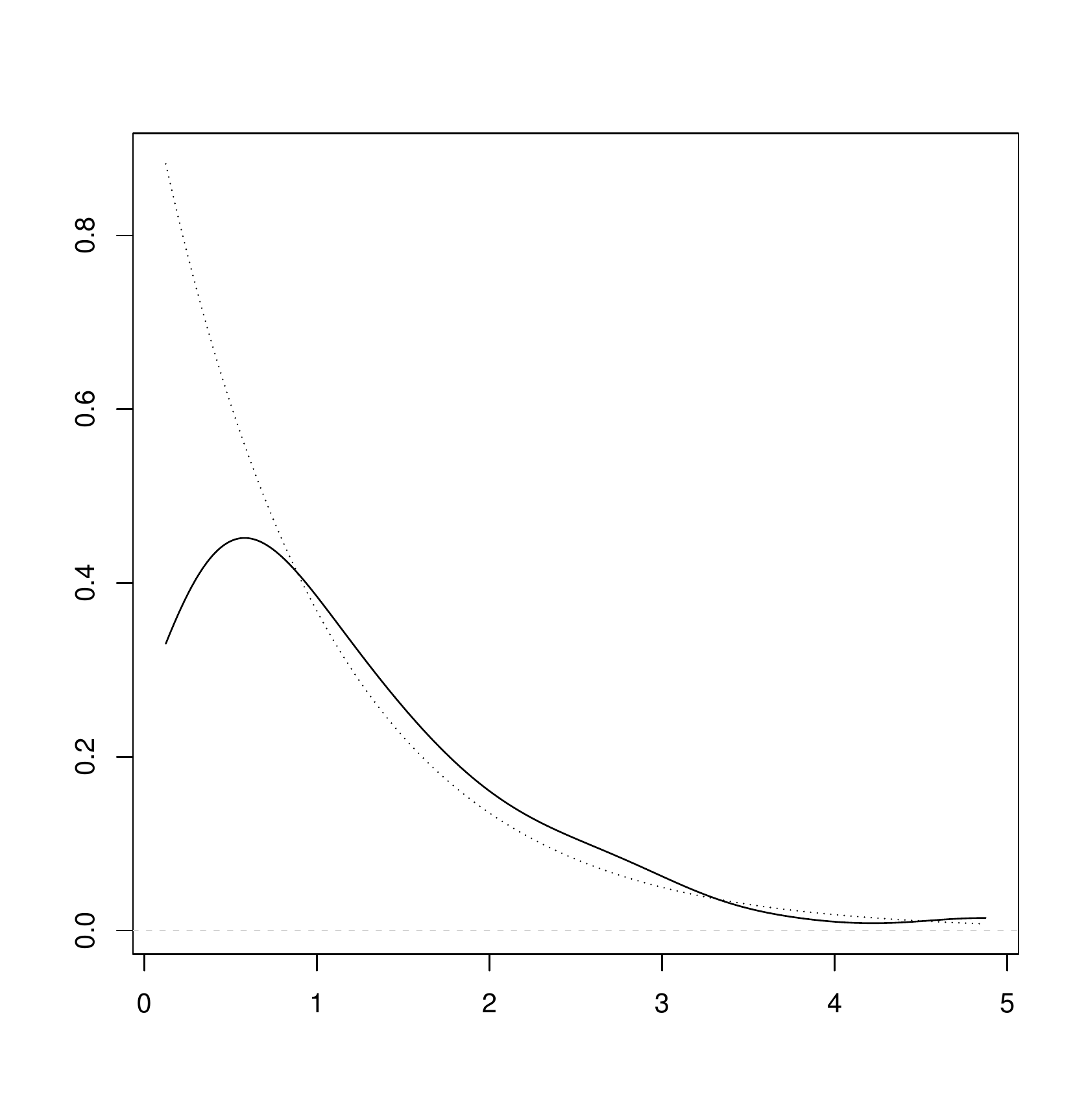}
\end{subfigure}
\caption{Simulation results for a compound Poisson process with jump
  intensity 1 and jump sizes having an exponential distribution with
  parameter $1$. Left plot: obtained measures for various algorithms,
  the right plot: the theoretical exponential density and the smoothed version of
  $\dhat{\Lambda}_1$ measure.} 
\label{fig:exp}
\end{figure}

\paragraph{Gaussian compounding distribution.}
 Figure \ref{norm} takes up the important example of compound
Poisson processes with Gaussian jumps. Once again, the estimates
$\hat\Lambda_k$ improve as $k$ increases, and the combined method
gives an estimate similar to $\hat\Lambda_3$.

\begin{figure}[H]
\begin{subfigure}{.5\textwidth}
  \centering
  \includegraphics[width=.95\linewidth]{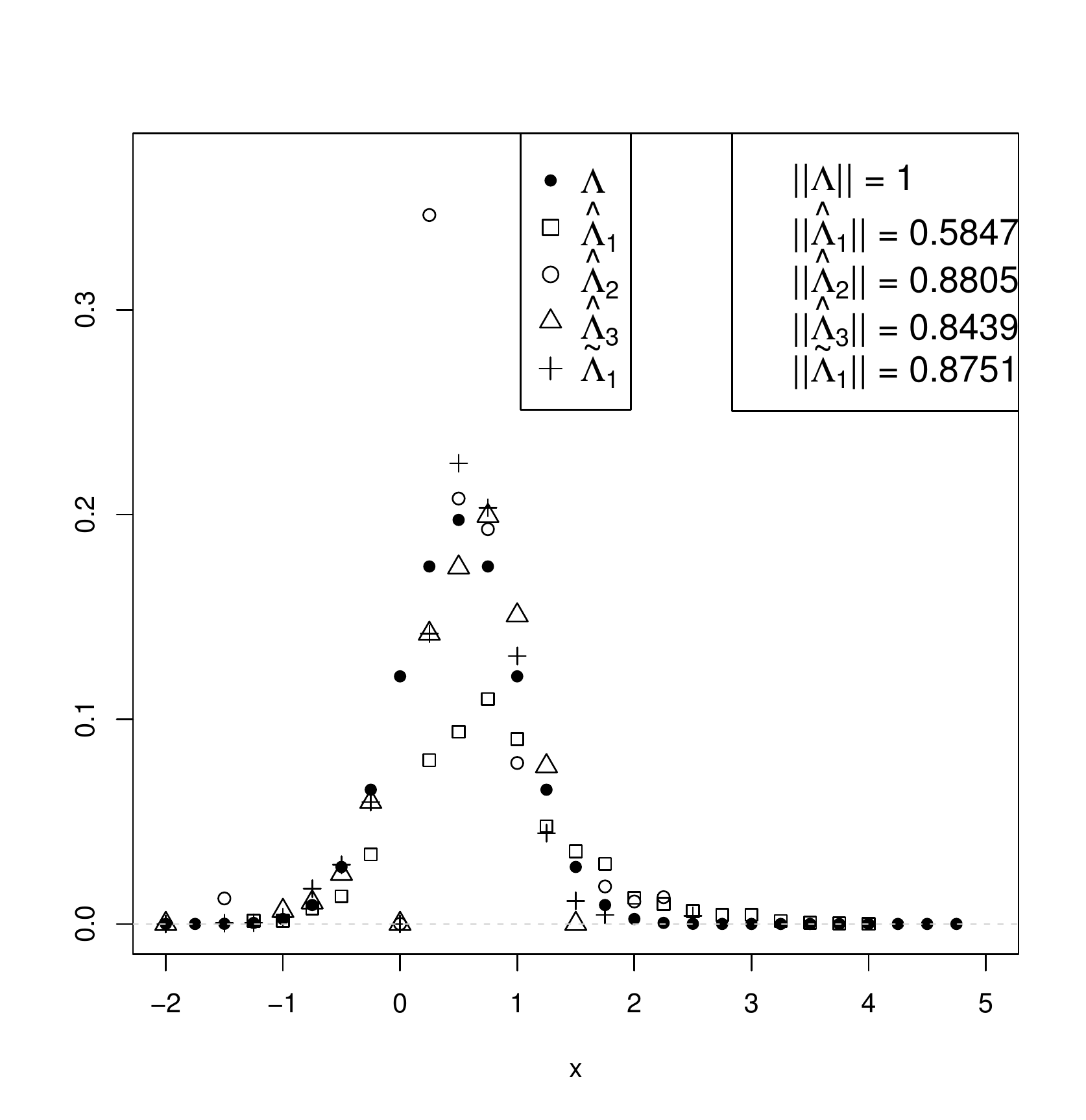}
\end{subfigure}%
\begin{subfigure}{.5\textwidth}
  \centering
  \includegraphics[width=.95\linewidth]{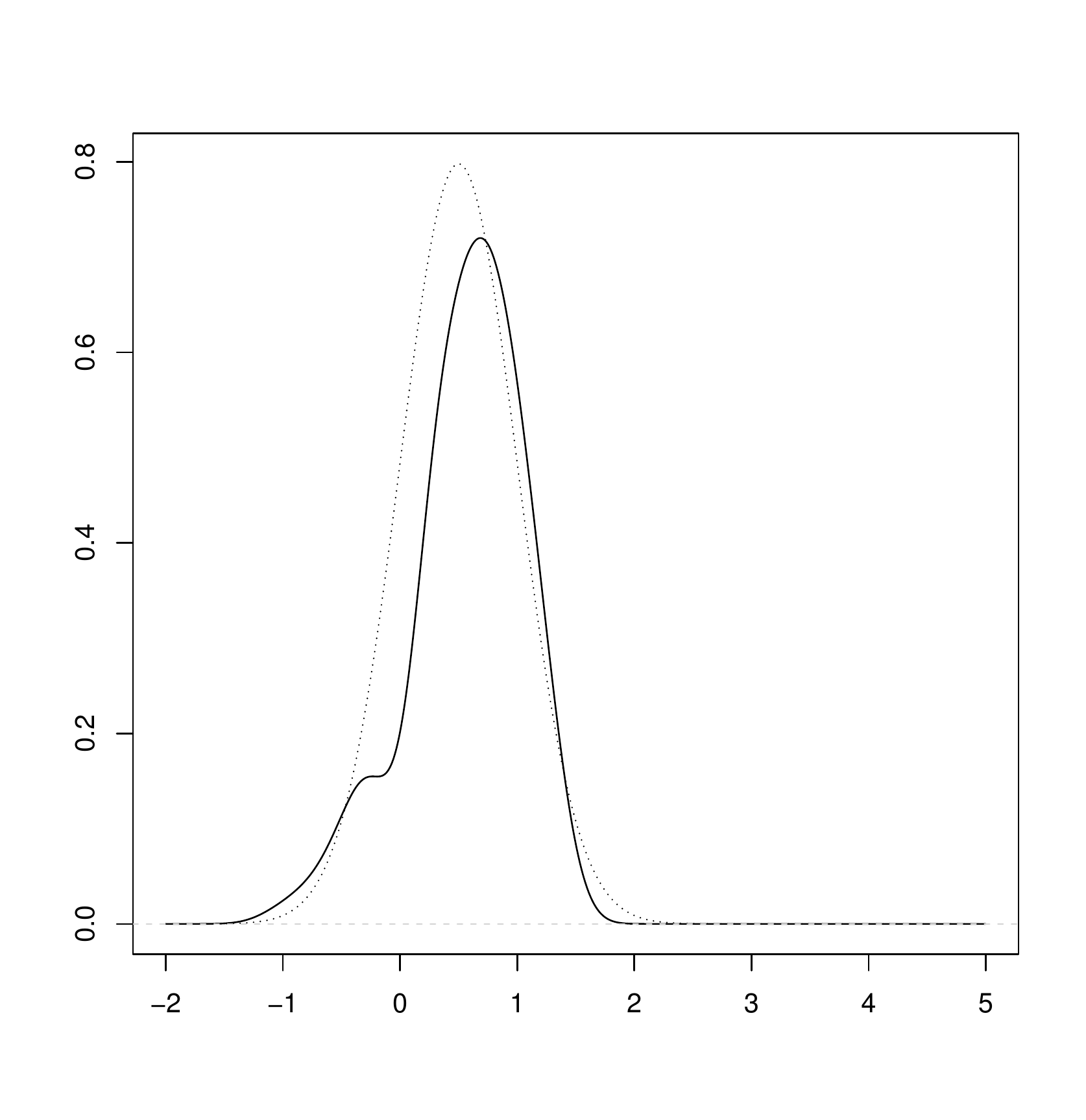}
\end{subfigure}
\caption{Left plot: Estimated compounding measure for a simulated sample with
  jump sizes having the standard Normal distribution with mean 0.5
  and variance 0.25. Right plot: the theoretical Gaussian density and
  the smoothed version of $\hat{\Lambda}_3$ measure.}
\label{norm}
\end{figure}

\section{Discussion}
\label{sec:discussion}
In this paper we proposed and analysed new algorithms based on the
characteristic function fitting (ChF) and convoluted cumulative distribution function fitting
(CoF) for non-parametric inference of the compounding measure of a
pure-jump L\'evy process. The algorithms are based on the recently
developed variational analysis of functionals of measures and the
corresponding steepest descent methods for constraint optimisation on
the cone of measures. CoF methods are capable of producing very
accurate estimates, but at the expense of growing computational
complexity. The ChF method critically depends on the initial
approximation measure due to highly irregular behaviour of the
objective function. We have shown that the problems of convergence of
the ChF algorithms can often be effectively overcome by choosing the
sample measure (discretised to the grid) as the initial approximation
measure. However, a better alternative, as we demonstrated in the paper,
is to use the measure obtained by the simplest ($k=1$) CoF
algorithm. This combined CoF--ChF algorithm is fast and in majority of
cases produces a measure which is closest in the total variation to
the measure under estimation and thus this is our method of choice.

The practical experience we gained during various tests allows us to
conclude that the suggested methods are especially well suited for
estimation of discrete jump size distributions. They work well even
with jumps that take both positive and negative values not necessarily
belonging to a regular lattice, demonstrating a clear advantage over
the existing methods, see~\cite{Buchman2003}, \cite{Buchman2004}. Use
of our algorithms for continuous compounding distributions require
more trial and error in choosing the right discretisation grid and
smoothing procedures to produce good results which should be then
compared or complemented to the direct methods of the density
estimation like in~\cite{Es2007}, \cite{WatKul:03}.

\section*{Acknowledgements} 
This work was supported by Swedish Vetenskapr\aa det grant
no.~11254331. The authors are grateful to Ilya Molchanov for fruitful discussions.



\end{document}